\documentclass[a4paper]{article}

\usepackage[pages=all, color=black, position={current page.south}, placement=bottom, scale=1, opacity=1, vshift=5mm]{background}

\usepackage[margin=1in]{geometry} % full-width

\usepackage{amsmath}
\usepackage{amsthm}
\usepackage{amssymb}

\usepackage[utf8]{inputenc}
\usepackage{hyperref}

\usepackage{graphicx}
\usepackage{verbatim,xcolor}
\usepackage{bm}
\usepackage{multirow}
\usepackage{algorithmicx}
\usepackage[ruled]{algorithm}
\usepackage{algpseudocode}
\usepackage[symbol]{footmisc}

\usepackage[sort&compress,numbers,square]{natbib}
\theoremstyle{plain}
\newtheorem{theorem}{Theorem}[section]

\newtheorem{lemma}[theorem]{Lemma}

\theoremstyle{definition}
\newtheorem{definition}[theorem]{Definition}

\newtheorem{remark}[theorem]{Remark}

\usepackage{graphicx, color}
\graphicspath{{fig/}}
\usepackage{algorithm, algpseudocode}
\usepackage{mathrsfs}

\usepackage{lipsum}
\numberwithin{equation}{section}

\newcommand{\ba}{{\bf a}}
\newcommand{\bb}{{\bf b}}

\newcommand{\be}{{\bf e}}
\newcommand{\bx}{{\bf x}}
\newcommand{\bn}{{\bf n}}
\newcommand{\bu}{{\bf u}}
\newcommand{\bv}{{\bf v}}

\newcommand{\bw}{{\bf w}}

\newcommand{\bV}{{\bf V}}

\newcommand{\bC}{{\bf C}}
\newcommand{\bD}{{\bf D}}

\newcommand{\bbf}{{\bf f}}

\newcommand{\bzero}{\bf{0}}

\newcommand{\Pih}{\Pi_h}
\newcommand{\Pz}{\mathcal{P}_0}

\newcommand{\cR}{\mathcal{R}}

\newcommand{\jump}[1]{[ #1 ]}
\newcommand{\avg}[1]{\{ #1\}}
\newcommand{\Hone}{H^1(\Omega)}

\newcommand{\Honezd}{[H^1_0(\Omega)]^d}
 
 \newcommand{\Ltwoz}{L^2_0(\Omega)}
 \newcommand{\Vh}{\bV_h}
 
  \newcommand{\Ch}{\bC_h}
  \newcommand{\Dh}{\bD_h}
\newcommand{\Th}{\mathcal{T}_h}

\newcommand{\Eh}{\mathcal{E}_h}
\newcommand{\Eho}{\mathcal{E}_h^o}
\newcommand{\Ehb}{\mathcal{E}_h^b}
\newcommand{\bne}{\bn_e}

\newcommand{\norm}[1]{\lVert #1\rVert}
\newcommand{\enorm}[1]{\lVert #1\rVert_{\mathcal{E}}}

\newcommand{\trinorm}[1]{{\left\vert\kern-0.25ex\left\vert\kern-0.25ex\left\vert #1 \right\vert\kern-0.25ex\right\vert\kern-0.25ex\right\vert}}

\newcommand{\buh}{\bu_h}
\newcommand{\ph}{p_h}

\newcommand{\half}{\frac{1}{2}}

\title{A low-cost, parameter-free, and pressure-robust enriched Galerkin method for the Stokes equations\footnote{Submitted to the editors in 2022.}}
\author{Seulip Lee,\thanks{Department of Mathematics, University of Georgia, Athens, GA 30602 (\texttt{seulip.lee@uga.edu})}
\and Lin Mu,\thanks{Department of Mathematics, University of Georgia, Athens, GA 30602 (\texttt{linmu@uga.edu})}}

\date{%	\today
}

\begin{document}
	\maketitle
	
	\begin{abstract}
		In this paper, we propose a low-cost, parameter-free, and pressure-robust Stokes solver based on the enriched Galerkin (EG) method with a discontinuous velocity enrichment function. The EG method employs the interior penalty discontinuous Galerkin (IPDG) formulation to weakly impose the continuity of the velocity function. However, the symmetric IPDG formulation, despite of its advantage of symmetry, requires a lot of computational effort to choose an optimal penalty parameter and to compute different trace terms. In order to reduce such effort, we replace the derivatives of the velocity function with its weak derivatives computed by the geometric data of elements. Therefore, our modified EG (mEG) method is a parameter-free numerical scheme which has reduced computational complexity as well as optimal rates of convergence. Moreover, we achieve pressure-robustness for the mEG method by employing a velocity reconstruction operator on the load vector on the right-hand side of the discrete system. The theoretical results are confirmed through numerical experiments with two- and three- dimensional examples.
	\end{abstract}

%%%%%%%%%%%%%%%%%%%%%%%%%%%%%%%%%%
% 1. Introduction
%%%%%%%%%%%%%%%%%%%%%%%%%%%%%%%%%%

\section{Introduction}
\label{sec:intro}

We consider the Stokes equations in a bounded domain $\Omega\subset \mathbb{R}^d$ for $d=2,3$ with simply connected Lipschitz boundary $\partial\Omega$: Find fluid velocity $\bu:\Omega\rightarrow\mathbb{R}^d$ and pressure $p:\Omega\rightarrow\mathbb{R}$ such that
\begin{subequations}\label{sys: governing}
\begin{alignat}{2}
-\nu\Delta \bu + \nabla p & = \bbf && \quad \text{in } \Omega, \label{eqn: governing1} \\
\nabla\cdot \bu &= 0 && \quad \text{in } \Omega,  \label{eqn: governing2}\\
\bu & = 0 && \quad \text{on }\partial \Omega, \label{eqn: governing3}
\end{alignat}
\end{subequations}
where $\nu>0$ is a constant fluid viscosity, and $\bbf$ is a given body force.

In the finite element framework, finite dimensional velocity and pressure spaces must satisfy the discrete inf-sup stability condition \cite{Ladyzhenskaya69,Babuska73,Brezzi74} to guarantee the well-posedness of the discrete problem corresponding to \eqref{sys: governing}.
Various mixed finite element methods (FEMs) have been developed under the discrete inf-sup condition, such as conforming and non-conforming mixed FEMs \cite{TaylorHood73,BernardiRaugel85,CrouzeixRaviart73}, discontinuous Galerkin methods 
\cite{HansboLarson08,GiraultEtAl05}, weak Galerkin methods \cite{wang2016weak,MuWYZ18}, and enriched Galerkin methods \cite{ChaabaneEtAl18,YiEtAl22-Stokes}.
These methods have been widely used for numerical simulations of the Stokes equations while providing their different advantages.

For the Stokes equations, discontinuous Galerkin (DG) methods have received attention as advanced numerical methods which have locally conservative divergence-free condition and geometric flexibility on meshes.
The interior penalty discontinuous Galerkin (IPDG) method is an example of DG methods, and it employs penalties to impose weakly the continuity of the solutions and boundary conditions.
The penalty formulation has been also adopted in enriched Galerkin methods for the Poisson equation \cite{sun2009locally,lee2016locally} and $C^0$ interior penalty methods for the biharmonic equation \cite{brenner2005c}.
Even though the IPDG method has been widely applied in numerical PDE solvers, it has been criticized for the difficulty of choosing proper penalty parameters.
It is well-known that a sufficiently large penalty parameter is required to ensure the stability in the symmetric IPDG method.
However, in numerical simulations, a large penalty parameter may cause the increased condition number of the stiffness matrix, which leads to inaccurate simulation results.
Also, the lower bounds for penalty parameters \cite{epshteyn2007estimation,ainsworth2007posteriori,ainsworth2010fully} does not seem practical for general meshes because the bounds depend on the angles of the mesh elements.
Therefore, we pay special attention to constructing a parameter-free scheme to resolve the difficulty on penalty parameters. 
Various parameter-free DG methods have been introduced for second-order elliptic problems by introducing extra degrees of freedom on edges/faces and auxiliary variables, e.g., hybrid high-order (HHO) methods \cite{di2015hybrid}, hybridizable discontinuous Galerkin (HDG) methods \cite{cockburn2009unified}, and weak Galerkin (WG) methods \cite{wang2013weak}.
By rewriting DG basis functions in the WG framework, another parameter-free DG method \cite{wang2014modified}, which is called a modified WG method, has been developed without increasing degrees of freedom.
Our work is inspired by this idea.

Our main goal in this paper is to develop a low-cost and parameter-free Stokes solver with the optimal rates in convergence.
The enriched Galerkin (EG) velocity and pressure spaces have been presented in 
\cite{YiEtAl22-Stokes} for solving the Stokes equations with minimal number of degrees of freedom.
The velocity space consists of linear Lagrange polynomials enriched by a discontinuous, piecewise linear, and mean-zero vector function per element, while the pressure is approximated by piecewise constant functions.
That is, a velocity function $\bv$ can be expressed as $\bv=\bv^C+\bv^D$, where $\bv^C$ is a continuous linear Lagrange polynomial and $\bv^D$ is a discontinuous piecewise linear enrichment function.
Compared to the previous EG method \cite{YiEtAl22-Stokes} using the IPDG formulation, our modified EG (mEG) method is developed
by replacing the derivatives of velocity functions with their weak derivatives \cite{mu2015modified}.
The weak derivatives are locally computed in each element by integration by parts using the interior function $\bv$ and the average of $\bv$ along edges/faces (details will be provided in Section~\ref{sec:EGwithWG}).
The weak derivatives for $\bv^C$ remain the same as $\nabla\bv^C$ and $\nabla\cdot\bv^C$.
For the discontinuous components $\bv^D$, the weak derivatives are computed as piecewise constant functions by using the geometric data of each element, e.g., vertices, edges/faces, and area/volume.
In the mEG method, the bilinear forms are simply assembled by the $L^2$-inner product of the weak derivatives and a parameter-free penalty term.
The other trace terms in the IPDG formulation are not needed.
Thus, the mEG method is parameter-free, and its implementation is guaranteed to require reduced computational complexity.
In the theoretical part, the coercivity and continuity of the bilinear form for the diffusion term in \eqref{eqn: governing1} hold true with no penalty parameter.
Since the bilinear form for the divergence term \eqref{eqn: governing2} remains the same as the previous EG method, the discrete inf-sup condition of the mEG method can be inherited from the previous one.
Through two- and three-dimensional examples, we compare
the numerical performance of our modified EG method and the previous EG method with different penalty parameters.
The numerical results demonstrate that our mEG method shows uniform stability and outperforms the previous method.

Pressure-robustness is an important property of numerical methods for the Stokes equations in the case of small viscosity $\nu\ll1$.
In the case, inf-sup stable pairs may not guarantee accurate numerical velocity solutions.
More precisely, in standard mixed FEMs including the EG method \cite{YiEtAl22-Stokes}, 
the velocity error bounds are coupled with a pressure term which is inversely proportional to the viscosity $\nu$.
Thus, the numerical simulation for velocity may be destroyed by the factor $1/\nu$.
In contrast, pressure-robust schemes can eliminate the pressure term from the velocity error bounds in the error estimates, so they guarantee accurate numerical velocity and pressure simultaneously.
In some mixed FEMs, the pressure-robustness has been achieved by applying a velocity reconstruction operator \cite{Linke12} to the load vector on the right hand side
(see \cite{LinkeMerdon16,GaugerLS19,Mu20,MuYZ21(2),LiZikatanov22,zhao2020pressure,HuLeeMuYi} as examples).
To develop a pressure-robust scheme corresponding to the mEG method, we employ the velocity reconstruction operator \cite{HuLeeMuYi} mapping the velocity test function into the first-order Brezzi-Douglas-Marini space.
Therefore, the pressure-robustness in the mEG method is achieved without compromising the optimal rates in convergence.

The remaining sections of this paper are structured as follows:
Some important definitions, notations, and trace properties are introduced in Section \ref{sec:prelim}.
In Section \ref{sec:EGwithWG}, we recall the EG method \cite{YiEtAl22-Stokes} and propose the modified EG (mEG) method without a penalty parameter.
In Section \ref{sec:wellerror}, we prove well-posedness and error estimates of our mEG method.
A pressure-robust mEG method is presented and its robust error estimates are proved in Section \ref{sec:PRmEG}.
In Section \ref{sec:nume}, we validate our theoretical results through numerical experiments in two and three dimensions.
We summarize our contribution in this paper and discuss related future research in Section \ref{sec:conclusion}.

%%%%%%%%%%%%%%%%%%%%%%%%%%%%%%%%%%
% 2. Preliminaries
%%%%%%%%%%%%%%%%%%%%%%%%%%%%%%%%%%

\section{Preliminaries}
\label{sec:prelim}
To begin with, we introduce some notations and definitions used throughout this paper.
For a bounded Lipschitz domain $\mathcal{D}\in\mathbb{R}^d$, where $d=2,3$, we denote the Sobolev space as $H^s(\mathcal{D})$ for a real number $s\geq 0$.
Its norm and seminorm are denoted by $\|\cdot\|_{s,\mathcal{D}}$ and $|\cdot|_{s,\mathcal{D}}$, respectively.
The space $H^0(\mathcal{D})$ coincides with $L^2(\mathcal{D})$, and the $L^2$-inner product is denoted by $(\cdot,\cdot)_\mathcal{D}$.
When $\mathcal{D}=\Omega$, the subscript $\mathcal{D}$ will be omitted.
These notations are generalized to vector- and tensor-valued Sobolev spaces.
The notation $H_0^1(\mathcal{D})$ means the space of $v\in H^1(\mathcal{D})$ such that $v=0$ on $\partial\mathcal{D}$, and $L_0^2(\mathcal{D})$ means the space of $v\in L^2(\mathcal{D})$ such that $(v,1)_\mathcal{D}=0$.
The polynomial spaces of degree less than or equal to $k$ are denoted as $P_k(\mathcal{D})$.

For discrete schemes, we assume that there exists a shape-regular triangulation $\Th$ of $\Omega$ whose elements $T\in \Th$ are triangles in two dimensions and tetrahedrons in three dimensions.
Then, $\Eh$ denotes the collection of all edges/faces in $\Th$, and $\Eh=\Eho\cup\Ehb$, where $\Eho$ is the collection of all the interior edges/faces and $\Ehb$ is that of the boundary edges/faces.
For each element $T\in\Th$, let $h_T$ denote the diameter of $T$ and $\bn_T$ (or $\bn$) denote the outward unit normal vector on $\partial T$.
For each interior edge/face $e\in \Eho$ shared by two adjacent elements $T^+$ and $T^-$, we let $\bn_e$ be the unit normal vector from $T^+$ to $T^-$.
For each $e\in\Ehb$, $\bn_e$ denotes the outward unit normal vector on $\partial \Omega$.

In a shape-regular triangulation $\Th$, the broken Sobolev space is defined as
\begin{equation*}
    H^s(\Th)=\{v\in L^2(\Omega):v|_T\in H^s(T),\ \forall T\in\Th\},
\end{equation*}
equipped with the norm
\begin{equation*}
    \norm{v}_{s,\Th}=\left(\sum_{T\in\Th} \norm{v}^2_{s,T}\right)^{1/2}.
\end{equation*}
When $s=0$, the $L^2$-inner product on $\Th$ is denoted by $(\cdot,\cdot)_{\Th}$.
Also, the $L^2$-inner product on $\Eh$ is denoted as $\langle\cdot,\cdot\rangle_{\Eh}$, and the $L^2$-norm on $\Eh$ is defined as
\begin{equation*}
    \norm{v}_{0,\Eh}=\left(\sum_{e\in\Eh} \norm{v}^2_{0,e}\right)^{1/2}.
\end{equation*}
The piecewise polynomial space corresponding to the broken Sobolev space is defined as
\begin{equation*}
    P_k(\Th) = \{v\in L^2(\Omega): v|_T\in P_k(T),\ \forall T\in\Th\}.
\end{equation*}
In addition, the jump and average of $v$ on $e\in \Eh$ are defined as
\begin{equation*}
    \jump{v}=\left\{\begin{array}{cl}
        v^+-v^- & \text{on}\ e\in \Eho, \\
        v & \text{on}\ e\in\Ehb,
    \end{array}\right.
    \quad
    \avg{v}=\left\{\begin{array}{cl}
        (v^++v^-)/2 & \text{on}\ e\in \Eho, \\
        v & \text{on}\ e\in\Ehb,
    \end{array}\right.
\end{equation*}
where $v^{\pm}$ is the trace of $v|_{T^\pm}$ on $e\in \partial T^+\cap\partial T^-$. 
These definitions are extended to vector- and tensor-valued functions.

We also introduce the trace properties mainly used in this paper. For any vector function $\bv$ and scalar function $q$, we have
\begin{equation}\label{eqn: jump-avg}
    \sum_{T\in\Th} \langle\bv\cdot\bn,q\rangle_{\partial T} = \langle \jump{\bv}\cdot\bne,\avg{q}\rangle_{\Eh}+\langle \avg{\bv}\cdot \bne,\jump{q}\rangle_{\Eho}.
\end{equation}
For any function $v\in H^1(T)$, the following trace inequality holds
\begin{equation}
    \norm{v}_{0,e}^2\leq C\left(h_T^{-1}\norm{v}_{0,T}^2+h_T\norm{\nabla v}_{0,T}^2\right).\label{eqn: trace}
\end{equation}

%%%%%%%%%%%%%%%%%%%%%%%%%%%%%%%%%%
% 3. A Modified Enriched Galerkin Method
%%%%%%%%%%%%%%%%%%%%%%%%%%%%%%%%%%

\section{A Modified Enriched Galerkin Method}
\label{sec:EGwithWG}
We consider the weak formulation for the Stokes problem \eqref{sys: governing}: Find $(\bu, p) \in \Honezd \times \Ltwoz$ such that
\begin{subequations}\label{sys: weak}
\begin{alignat}{2}
\nu   (\nabla \bu, \nabla \bv) - (\nabla \cdot \bv, p) & = (\bbf, \bv), && \quad \forall \bv \in  \Honezd, \label{eqn: weak1} 
\\
(\nabla \cdot \bu, q) & = 0, && \quad \forall q \in  \Ltwoz. \label{eqn: weak2}
\end{alignat}
\end{subequations}
We recall the EG method \cite{YiEtAl22-Stokes} with its finite dimensional velocity and pressure spaces, and then introduce weak derivatives to establish the modified EG method in this section.

%%%%%%%%%% 3.1 %%%%%%%%%%%

\subsection{Standard enriched Galerkin method with interior penalty}

We first introduce the EG finite dimensional velocity and pressure spaces.
Let us denote the space of continuous components for velocity as
\begin{equation*}
\Ch = \{\bv^C \in \Honezd : \bv^C|_{T} \in [P_1(T)]^d,\ \forall T \in \Th \}.
\end{equation*}
The space of discontinuous components for velocity is defined as
\begin{equation*}
    \Dh = \{\bv^D \in L^2(\Omega) : \bv^D|_{T} = c (\bx - \bx_T),\ c \in \mathbb{R},\ \forall T \in \Th\},
\end{equation*}
where $\bx_T$ is the barycenter of $T\in\Th$.
Then, the EG finite dimensional velocity space is defined as
\begin{equation*}
    \Vh = \Ch \oplus \Dh,
\end{equation*}
that is, any function $\bv\in\Vh$ consists of unique continuous and discontinuous components, $\bv=\bv^C+\bv^D$ for $\bv^C\in\Ch$ and $\bv^D\in\Dh$.
At the same time, the EG pressure space is chosen as
\begin{equation*}
    Q_h = \{ q \in \Ltwoz : q|_T \in P_0(T),\ \forall T \in \Th \}.
\end{equation*}
Therefore, the EG method \cite{YiEtAl22-Stokes} is formulated with the pair of the spaces $\Vh\times Q_h$.
\begin{algorithm}[H]
\caption{Enriched Galerkin (\texttt{EG}) method} \label{alg:EG}
Find $( \bu_h, \ph) \in \Vh \times Q_h $ such that
\begin{subequations}\label{sys:EG}
\begin{alignat}{2}
\ba(\buh,\bv)  - \bb(\bv, \ph) &= (\bbf,  \bv), &&\quad \forall \bv \in\Vh, \label{eqn: stand-eg1}\\
\bb(\buh,q) &= 0, &&\quad \forall q\in Q_h,  \label{eqn: stand-eg2}
\end{alignat}
\end{subequations}
where 
\begin{subequations}\label{sys: bilinear_IPDG}
\begin{align}
\ba(\bw,\bv) &:= \nu \big((\nabla \bw,\nabla \bv)_{\Th} -  \langle \avg{\nabla\bw}\cdot \bn_e, \jump{\bv} \rangle_{\Eh} \nonumber\\
&\qquad\qquad\qquad\qquad- \langle \avg {\nabla\bv}\cdot \bn_e,\jump{\bw}\rangle_{\Eh}  +  \rho \langle h_e^{-1}\jump{\bw},
\jump{\bv}\rangle_{\Eh} \big), \label{eqn: bia} \\
\bb(\bw,q)&:= (\nabla\cdot\bw, q)_{\Th} - \langle \jump{\bw}\cdot\bn_e,\avg{q} \rangle_{\Eh}.  \label{eqn: bib}
\end{align}
\end{subequations}
Here, $\rho >0$ is a penalty parameter and $h_e = |e|^{1/(d-1)}$, where $|e|$ is the length/area of the edge/face $e \in \Eh$. 
\end{algorithm}

In the EG method,
the interior penalty discontinuous Galerkin (IPDG) formulation is adopted to weakly impose the continuity of the discontinuous component $\bv^D\in \Dh$, and it requires a sufficiently large penalty parameter $\rho$ to guarantee the well-posedness of the method.

%%%%%%%%%% 3.2 %%%%%%%%%%%

\subsection{Modified enriched Galerkin method with weak derivatives}

We introduce a weak Galerkin (WG) finite element space for velocity \cite{wang2016weak},
\begin{equation*}
    \bm{\mathcal{V}}_h=\{\bm{\upsilon}=\{\bm{\upsilon}_0,\bm{\upsilon}_b\}\mid\bm{\upsilon}_0|_T\in [P_1(T)]^d,\ \forall T\in\Th,\ \bm{\upsilon}_b|_e\in[P_1(e)]^d,\ \forall e\in \Eh\}.
\end{equation*}
Then, the EG velocity $\bv\in\Vh$ can be viewed as a WG function in $\bm{\mathcal{V}}_h$, that is,
\begin{equation*}
    \bm{\upsilon}_0=\bv, \quad \bm{\upsilon}_b=\avg{\bv}\quad\Rightarrow\quad
    \{\bv,\avg{\bv}\}\in\bm{\mathcal{V}}_h,
\end{equation*}
and the weak derivatives for $\bv\in\Vh$ are locally defined as follows.
\begin{definition}
The weak gradient operator \cite{mu2015modified} is defined as $\left.\nabla_w \bm{\upsilon}\right|_T\in[P_0(T)]^{d\times d}$ when $\bm{\upsilon}=\{\bm{\upsilon}_0,\bm{\upsilon}_b\}\in\bm{\mathcal{V}}_h$ satisfying
    \begin{equation*}
        (\nabla_w \bm{\upsilon},\aleph)_T=\langle \bm{\upsilon}_b,\aleph\cdot\bn\rangle_{\partial T},\quad\forall \aleph\in [P_0(T)]^{d\times d}.
    \end{equation*}
In a similar manner, the weak gradient for the EG velocity $\bv\in\Vh$ is defined as $\left.\nabla_w\bv\right|_T\in[P_0(T)]^{d\times d}$ such that
    \begin{equation*}
        (\nabla_w \bv,\aleph)_T=\langle \avg{\bv},\aleph\cdot\bn\rangle_{\partial T},\quad\forall \aleph\in [P_0(T)]^{d\times d}.
    \end{equation*}
Moreover, the weak divergence operator \cite{mu2015modified} for $\bv\in \Vh$ is defined as $\left.\nabla_w\cdot\bv\right|_T\in P_0(T)$ such that
\begin{equation*}
        (\nabla_w \cdot\bv,q)_T=\langle \avg{\bv}\cdot\bn,q\rangle_{\partial T},\quad\forall q\in P_0(T).
    \end{equation*}
\end{definition}
\begin{remark} For any EG velocity function $\bv\in \Vh$, the differences between the weak derivatives and regular derivatives are given as
\begin{subequations}\label{sys: relation}
\begin{alignat}{2}
\left(\nabla\bv-\nabla_w\bv,\aleph\right)_{\Th}&=\langle\jump{\bv},\avg{\aleph}\cdot\bn_e \rangle_{\Eh},&&\quad\forall \aleph\in [P_0(\Th)]^{d\times d},\label{eqn: relation1}\\
\left(\nabla\cdot \bv - \nabla_w\cdot \bv,q\right)_{\Th}&=\langle [\bv]\cdot\bn_e,\avg{q}\rangle_{\Eh}, &&\quad\forall q\in P_0(\Th).\label{eqn: relation2}
\end{alignat}
\end{subequations}
These identities are simply obtained from the definition of the weak derivatives and integration by parts.
Since the EG velocity consists of $\bv^C\in \Ch$ and $\bv^D\in\Dh$, it is clear to see from \eqref{sys: relation} that $\nabla_w\bv^C=\nabla\bv^C$, $\nabla_w\cdot\bv^C=\nabla\cdot\bv^C$, and the jumps of $\bv^D$ on $e\in\Eho$ cause the differences.
In practice, the weak gradient $\nabla_w\bv^D$ is locally determined by
\begin{equation*}
    (\nabla_w\bv^D)_{i,j} = \frac{n_j}{|T|}\langle \avg{\bv^D}, \mathbf{e}_i\rangle_{\partial T},\quad 1\leq i,j\leq d,
\end{equation*}
where $n_j$ is the $j$-th component of $\bn$ and $\mathbf{e}_i$ is the standard unit vector whose $i$-th component is 1.
Since $\bv^D|_T=c(\bx-\bx_T)$ is a linear function, the above line/surface integral can be simply computed by the one-point quadrature rule on each edge/face, respectively.
Also, the weak divergence $\nabla_w\cdot\bv^D$ is the trace of $\nabla_w\bv^D$ from the definition, which implies no associated cost in computing the weak divergence.
\end{remark}

Therefore, we propose the modified enriched Galerkin method which is formulated by the weak derivatives for the EG velocity $\bv\in\Vh$. 
\begin{algorithm}[H]
\caption{Modified enriched Galerkin (\texttt{mEG}) method} \label{alg:mEG}
Find $( \bu_h, \ph) \in \Vh \times Q_h $ such that
\begin{subequations}\label{sys: mEG}
\begin{alignat}{2}
\ba_w(\buh,\bv)  - \bb_w(\bv, \ph) &= (\bbf,  \bv), &&\quad \forall \bv \in\Vh, \label{eqn: mEG1}\\
\bb_w(\buh,q) &= 0, &&\quad \forall q\in Q_h,  \label{eqn: mEG2}
\end{alignat}
\end{subequations}
where
\begin{subequations}\label{sys: bilinear_WG}
\begin{alignat}{2}
\ba_w(\bw,\bv) &:= \nu \big((\nabla_w \bw,\nabla_w \bv)_{\Th}  +  \langle h_e^{-1}\jump{\bw},
\jump{\bv}\rangle_{\Eh} \big), \label{eqn: biaw} \\
\bb_w(\bw,q) &:=(\nabla_w\cdot\bw,q)_{\Th}.  \label{eqn: bibw}
\end{alignat}
\end{subequations}

In this case, $h_e = |e|^{1/(d-1)}$, where $|e|$ is the length/area of the edge/face $e \in \Eh$.
\end{algorithm}

\begin{remark}
There is no penalty parameter in the \texttt{mEG} method, while the \texttt{EG} method in Algorithm~\ref{alg:EG} requires a sufficiently large penalty parameter $\rho$.
We observe the result of applying \eqref{eqn: relation1} to \eqref{eqn: biaw},
\begin{align*}
    \ba_w(\bw,\bv)&= \nu \big((\nabla \bw,\nabla \bv)_{\Th} -  \langle \avg{\nabla\bw} \bn_e, \jump{\bv} \rangle_{\Eh} \nonumber\\
&\qquad\qquad\qquad\qquad- \langle \avg {\nabla_w\bv} \bn_e,\jump{\bw}\rangle_{\Eh}  +  \langle h_e^{-1}\jump{\bw},\jump{\bv}\rangle_{\Eh} \big).
\end{align*}
By comparing with $\ba(\cdot,\cdot)$ in \eqref{eqn: bia}, the average of the gradient in the symmetric term of $\ba(\cdot,\cdot)$ is replaced by that of the weak gradient, and the bilinear form $\ba_w(\cdot,\cdot)$ does not depend on the penalty parameter $\rho$
(see also \cite{xie2020convergence}).
In addition, the identity \eqref{eqn: relation2} implies that for any $\bw\in \Vh$ and $q\in Q_h$,
\begin{equation}\label{b equalto bw}
    \bb_w(\bw,q)=\bb(\bw,q),
\end{equation}
which makes it simple to prove the discrete inf-sup condition.
In practice, this allows us to use the same block matrices corresponding to $\bb(\cdot,\cdot)$ (or $\bb_w(\cdot,\cdot)$) for both \texttt{EG} and 
\texttt{mEG} methods.
\end{remark}

\begin{remark}
For a non-homogeneous Dirichlet boundary condition ($\bu=\mathbf{g}$ on $\partial\Omega$), the EG velocity in \eqref{sys: mEG} satisfies $\bu_h=\bu_h^C+\bu_h^D=\mathbf{g}$ on any $e\in\Ehb$.
We treat $\bu_h^C = \mathbf{g}$ as an essential boundary condition and $\bu_h^D=\mathbf{0}$ as a natural boundary condition.
In the \texttt{mEG} method, the condition $\bu_h^D=\bzero$ is weakly applied to locally compute the weak derivatives in the elements adjoining the boundary. 
\end{remark}

%%%%%%%%%%%%%%%%%%%%%%%%%%%%%%%%%%
% 4. Well-Posedness and Error Analysis
%%%%%%%%%%%%%%%%%%%%%%%%%%%%%%%%%%

\section{Well-Posedness and Error Analysis}
\label{sec:wellerror}

For the \texttt{EG} method \cite{YiEtAl22-Stokes} in Algorithm~\ref{alg:EG}, the well-posedness and error estimates have been proved in terms of the energy norm in $\Vh$,
\begin{equation*}
    \enorm{\bv} := \left(\norm{\nabla \bv}_{0, \Th}^2 + \rho \norm{h_e^{-1/2}  \jump{\bv}}_{0, \Eh}^2\right)^\half.
\end{equation*}
To show the discrete inf-sup condition and a priori error estimates for the \texttt{mEG} method in Algorithm~\ref{alg:mEG}, we employ the theoretical results of the \texttt{EG} method.
In this case, the \texttt{mEG} method includes the weak derivatives, so it requires a mesh-dependent norm corresponding to the bilinear form $\ba_w(\cdot,\cdot)$,
\begin{equation*}
    \trinorm{\bv}:=\left(\norm{\nabla_w\bv}_{0,\Th}^2 + \norm{h_e^{-1/2}  \jump{\bv}}_{0, \Eh}^2 \right)^\half.
\end{equation*}
Then, the following norm equivalence helps to prove the theoretical results of the \texttt{mEG} method.
\begin{lemma}
For any $\bv\in\Vh$, there are positive constants $\gamma_*$ and $\gamma^*$ independent of $h:=\max_{T\in\Th}h_T$ such that
\begin{equation}
    \gamma_*\trinorm{\bv}\leq\enorm{\bv}\leq\gamma^*\trinorm{\bv}.\label{eqn: normeq}
\end{equation}
\end{lemma}
\begin{proof}
We start with the relation \eqref{eqn: relation1} while choosing $\aleph=\nabla_w\bv$,
\begin{equation*}
    \norm{\nabla_w\bv}_{0,\Th}^2=\left(\nabla_w \bv,\nabla_w\bv\right)_{\Th}=\left(\nabla \bv,\nabla_w\bv\right)_{\Th}-\langle \jump{\bv},\avg{\nabla_w\bv}\cdot\bn_e\rangle_{\Eho}.
\end{equation*}
Then, the first term is simply bounded using the Cauchy-Schwarz inequality,
\begin{equation*}
    \left(\nabla \bv,\nabla_w\bv\right)_{\Th}\leq \norm{\nabla\bv}_{0,\Th}\norm{\nabla_w\bv}_{0,\Th},
\end{equation*}
and the second term is bounded using the Cauchy-Schwarz inequality and trace inequality \eqref{eqn: trace},
\begin{align*}
    \langle \jump{\bv},\avg{\nabla_w\bv}\cdot\bn_e\rangle_{\Eho}&\leq\norm{h_e^{-1/2}\jump{\bv}}_{0,\Eho}\norm{h_e^{1/2}\avg{\nabla_w\bv}}_{0,\Eho}\\
    &\leq C\norm{h_e^{-1/2}\jump{\bv}}_{0,\Eho}\norm{\nabla_w\bv}_{0,\Th}.
\end{align*}
Hence, we arrive at
\begin{equation*}
    \norm{\nabla_w\bv}_{0,\Th}\leq C\enorm{\bv},
\end{equation*}
which implies the lower bound in \eqref{eqn: normeq}.

On the other hand, we choose $\aleph=\nabla\bv$ in \eqref{eqn: relation1} and apply the Cauchy-Schwarz inequality and \eqref{eqn: trace} to obtain
\begin{equation*}
    \norm{\nabla\bv}_{0,\Th}^2\leq\left(\nabla_w\bv,\nabla\bv\right)_{\Th}+\langle \jump{\bv},\avg{\nabla\bv}\cdot\bn_e\rangle_{\Eho}\leq C\trinorm{\bv}\norm{\nabla\bv}_{0,\Th}.
\end{equation*}
Therefore, it is clear to see that
\begin{equation*}
    \norm{\nabla\bv}_{0,\Th}\leq C\trinorm{\bv},
\end{equation*}
which yields the upper bound in \eqref{eqn: normeq}.
\end{proof}

%%%%%%%%%% 4.1 %%%%%%%%%%%

\subsection{Well-posedness}

In this subsection, with the norm equivalence \eqref{eqn: normeq}, we show the well-posedness of the \texttt{mEG} method by proving the essential properties of the bilinear forms.
\begin{lemma}
There exists a positive constant $C$ independent of $h$ such that
\begin{equation}
\sup_{\bv\in \Vh}\frac{\bb_w(\bv,q)}{\trinorm{\bv}}\geq C\norm{q}_0,\quad\forall q\in Q_h.\label{eqn: infsupb}
\end{equation}
\end{lemma}
\begin{proof}
It follows from
\eqref{b equalto bw}, \eqref{eqn: normeq}, and
the discrete inf-sup condition in \cite{YiEtAl22-Stokes}
that 
\begin{equation*}
    \sup_{\bv\in \Vh}\frac{\bb_w(\bv,q)}{\gamma_*\trinorm{\bv}}\geq\sup_{\bv\in \Vh}\frac{\bb(\bv,q)}{\enorm{\bv}}\geq C_{\mathcal{E}}\norm{q}_0,\quad\forall q\in Q_h,
\end{equation*}
where $C_\mathcal{E}$ is the constant for the inf-sup condition with respect to $\enorm{\cdot}$.
\end{proof}

It is also straightforward to show the continuity of $\bb_w(\cdot,\cdot)$ with respect to the norm $\trinorm{\cdot}$ using the norm equivalence \eqref{eqn: normeq}.
\begin{lemma}
For any $\bv\in\Vh$ and $q\in Q_h$, there exists a positive constant $C$ independent of $h$ satisfying
\begin{equation}
    |\bb_w(\bv,q)|\leq C\norm{q}_0\trinorm{\bv}.\label{eqn: contib}
\end{equation}
\end{lemma}
\begin{proof}
It follows from
\eqref{b equalto bw}, \eqref{eqn: normeq}, and
the continuity of $\bb(\cdot,\cdot)$ in \cite{YiEtAl22-Stokes} that
\begin{equation*}
    |\bb_w(\bv,q)|=|\bb(\bv,q)|\leq C\norm{q}_0\enorm{\bv}\leq C\gamma^*\norm{q}_0\trinorm{\bv}.
\end{equation*}
\end{proof}

In addition, we obtain the coercivity and continuity of $\ba_w(\cdot,\cdot)$ with respect to $\trinorm{\cdot}$.
(See \cite{mu2015modified} for details.)

\begin{lemma}
For any $\bv,\bw\in \Vh$, we have the coercivity and continuity results for $\ba_w(\cdot,\cdot)$:
\begin{align}
    &\ba_w(\bv,\bv)=\nu\trinorm{\bv}^2,\label{eqn: coera}\\
    &|\ba_w(\bv,\bw)|\leq \nu\trinorm{\bv}\trinorm{\bw}.\label{eqn: contia}
\end{align}
\end{lemma}

Thus, we obtain the well-posedness of the \texttt{mEG} method.
\begin{theorem}
There exists a unique solution $(\buh,\ph)\in \Vh\times Q_h$ to the \texttt{mEG} method in Algorithm~\ref{alg:mEG}. 
\end{theorem}
\begin{proof}
Since $\Vh$ and $Q_h$ are finite dimensional spaces, it suffices to show that $\bu_h=\mathbf{0}$ and $p_h=0$ when $\bbf=\mathbf{0}$.
If we choose $\bv=\bu_h$ in \eqref{eqn: mEG1} and $q=p_h$ in \eqref{eqn: mEG2} and add the two equations, then we obtain
\begin{equation*}
    \ba_w(\bu_h,\bu_h)=0.
\end{equation*}
Thus, it follows from \eqref{eqn: coera} that $\trinorm{\bu_h}=0$, so $\bu_h=\mathbf{0}$.
Moreover, the fact $\bu_h=\mathbf{0}$ in \eqref{sys: mEG} implies
\begin{equation*}
    \bb_w(\bv,p_h)=0,\quad \forall \bv\in \Vh.
\end{equation*}
Therefore, we have $\norm{p_h}_0=0$ from \eqref{eqn: infsupb}, which gives $p_h=0$.
\end{proof}

%%%%%%%%%% 4.2 %%%%%%%%%%%

\subsection{Error estimates}

We prove error estimates for velocity and pressure with respect to the mesh-dependent norm $\trinorm{\cdot}$ and the $L^2$-norm, respectively.
Let $\Pih: [H^2(\Omega)]^d \to \Vh$ be the interpolation operator \cite{yi2022locking} such that
\begin{equation*}
\Pi_h\bv=\Pi_h^C\bv+\Pi_h^D\bv,
\end{equation*}
where $\Pi_h^C\bv\in \bC_h$
is the nodal value interpolant of $\bv$ and $\Pi_h^D\bv\in \bD_h$ satisfies 
$(\nabla\cdot\Pi_h^D\bv,1)_T=(\nabla\cdot(\bv - \Pi_h^C \bv), 1)_{T}$ for all $T\in \Th$.
The corresponding interpolation error estimates \cite{yi2022locking} are as follows:
\begin{subequations}\label{sys: Pih}
\begin{alignat}{2}
& |\bv - \Pih \bv | _{j,\Th} \leq C h^{m-j} |\bv|_{m},&&\quad 0 \leq j \leq m \leq 2, \quad\forall\bv\in[H^2(\Omega)]^d, \label{eqn: Pih_err} \\
& \enorm{\bv - \Pih \bv} \leq C h \norm{\bv}_2, &&\quad\forall \bv \in [H^2(\Omega)]^d.  \label{eqn: Pih_energy_err}
\end{alignat}
\end{subequations}
We also introduce the local $L^2$-projection $\mathcal{P}_0: \Hone \to Q_h$ satisfying $(q - \Pz q, 1)_T = 0$ for all $T\in\Th$ and its error estimate,
\begin{equation}
  \norm{ q - 
  \Pz q}_0 \leq C h \norm{q}_1,  \quad \forall q \in H^1(\Omega).  \label{eqn: P_err}  
\end{equation}
Furthermore, let us denote $\Theta_h:[H^2(\Omega)]^d\to\bm{\mathcal{V}}_h$ as 
\begin{equation*}
    \Theta_h\bu = \{\Theta_0\bu,\Theta_b\bu\},
\end{equation*}
where $\Theta_0$ and $\Theta_b$ are the local $L^2$-projections onto $[P_1(T)]^d$ for all $T\in\Th$ and $[P_1(e)]^d$ for all $e\in\Eh$, respectively.
Then, we have the following commutative property \cite{wang2016weak},
\begin{equation}
    \nabla_w(\Theta_h\bv) = \bm{\Theta}_h(\nabla\bv),\label{eqn: projprop}
\end{equation}
where $\bm{\Theta}_h$ is the local $L^2$-projection onto $[P_0(T)]^{d\times d}$.

We define error functions used in the error estimates,
\begin{equation}\label{eqn: errfunctions}
    \bm{\chi}_h=\bu-\Pi_h\bu,\quad\mathbf{e}_h=\Pi_h\bu-\bu_h,\quad\xi_h=p-\Pz p,\quad\epsilon_h=\Pz p-p_h.
\end{equation}
Then, we derive the main error equations in the following lemma.
\begin{lemma}\label{lemma: erreqn}
For any $\bv\in \Vh$ and $q\in Q_h$, we have
\begin{subequations}\label{sys: erreqn}
\begin{alignat}{2}
\ba_w(\be_h,\bv)-\bb_w(\bv,\epsilon_h)&=l_1(\bu,\bv)+l_2(\bu,\bv)+\mathbf{s}(\Pi_h\bu,\bv)+\bb_w(\bv,\xi_h),\label{eqn: erreqn1}\\
\bb_w(\be_h,q)&=-\bb_w(\bm{\chi}_h,q),\label{eqn: erreqn2}
\end{alignat}
\end{subequations}
where the supplemental bilinear forms are defined as follows:
\begin{align*}
    &l_1(\bu,\bv)=\nu\sum_{T\in\Th}\langle \nabla\bu\cdot\bn-\bm{\Theta}_h(\nabla\bu)\cdot\bn,\bv-\avg{\bv}\rangle_{\partial T},\\
    &l_2(\bu,\bv)=\nu\left(\nabla_w(\Pi_h\bu-\Theta_h\bu),\nabla_w\bv\right)_{\Th},\\
    &\mathbf{s}(\Pi_h\bu,\bv)=\nu\langle h_e^{-1}\jump{\Pi_h\bu},\jump{\bv}\rangle_{\Eh}.
\end{align*}
\end{lemma}
\begin{proof}
For any $\bv\in\Vh$, integration by parts and the definition of $\bm{\Theta}_h$ imply
\begin{align*}
    \left(-\Delta\bu,\bv\right)_{\Th} &= -\sum_{T\in \Th}\langle\nabla \bu\cdot\bn,\bv\rangle_{\partial T} + \left(\nabla\bu,\nabla\bv\right)_{\Th}\\
    &=-\sum_{T\in \Th}\langle\nabla \bu\cdot\bn,\bv-\avg{\bv}\rangle_{\partial T}+\left(\bm{\Theta}_h(\nabla\bu),\nabla\bv\right)_{\Th}.
\end{align*}
Then, the definition of the weak gradient and the commutative property \eqref{eqn: projprop} lead to
\begin{align*}
    \left(\bm{\Theta}_h(\nabla\bu),\nabla\bv\right)_{\Th}&= \left(\bm{\Theta}_h(\nabla\bu),\nabla_w\bv\right)_{\Th}+\left(\bm{\Theta}_h(\nabla\bu),\nabla\bv-\nabla_w\bv\right)_{\Th}\\
    &=\left(\nabla_w(\Theta_h\bu),\nabla_w\bv\right)_{\Th}+\sum_{T\in\Th}\langle\bm{\Theta}_h(\nabla \bu)\cdot\bn,\bv-\avg{\bv}\rangle_{\partial T}.
\end{align*}
Hence, we obtain
\begin{align*}
    \left(-\Delta\bu,\bv\right)_{\Th}&=\left(\nabla_w(\Theta_h\bu),\nabla_w\bv\right)_{\Th}-\sum_{T\in\Th}\langle \nabla\bu\cdot\bn-\bm{\Theta}_h(\nabla\bu)\cdot\bn,\bv-\avg{\bv}\rangle_{\partial T},\\
    \left(\nabla p,\bv\right)_{\Th}&=\bb_w(\bv,p),
\end{align*}
where the second equation is obtained by
the trace identity \eqref{eqn: jump-avg}, the continuity of $p$, and \eqref{eqn: relation2}.
Then, by combining the above two equations in the equation \eqref{eqn: governing1}, we have
\begin{equation*}
    \left(\nabla_w(\Theta_h\bu),\nabla_w\bv\right)_{\Th}-\bb_w(\bv,p)=\left(\bbf,\bv\right)+l_1(\bu,\bv).
\end{equation*}
If we add proper terms including $\Pi_h\bu$ to both sides and subtract $\bb_w(\bv,P_0p)$ from both sides, we get
\begin{equation*}
    \ba_w(\Pi_h\bu,\bv)-\bb_w(\bv,\mathcal{P}_0p)=(\bbf,\bv)+l_1(\bu,\bv)+l_2(\bu,\bv)+\mathbf{s}(\Pi_h\bu,\bv)+\bb_w(\bv,\xi_h).
\end{equation*}
By comparing this equation with \eqref{eqn: mEG1} in the \texttt{mEG} method, we arrive at
\begin{equation*}
    \ba_w(\be_h,\bv)-\bb_w(\bv,\epsilon_h)=l_1(\bu,\bv)+l_2(\bu,\bv)+\mathbf{s}(\Pi_h\bu,\bv)+\bb_w(\bv,\xi_h).
\end{equation*}
Furthermore, the continuity of $\bu$ and \eqref{eqn: mEG2} imply
\begin{equation*}
    \left(\nabla\cdot\bu,q\right)_{\Th}=\bb_w(\bu,q)=0=\bb_w(\bu_h,q),
\end{equation*}
so \eqref{eqn: erreqn2} is obtained by subtracting $\bb_w(\Pi_h\bu,q)$ from both sides.
\end{proof}

We provide the upper bounds for the supplementary bilinear forms in Lemma~\ref{lemma: erreqn}.

\begin{lemma}\label{lemma: supplement estimates}
We assume that $\bw\in[H^2(\Omega)]^d$ and $\bv\in \Vh$. Then, we have
\begin{subequations}\label{sys: suppest}
\begin{alignat}{2}
& \left|l_1(\bw,\bv)\right|\leq C\nu h\norm{\bw}_2\trinorm{\bv},\label{eqn: suppest1} \\
& \left|l_2(\bw,\bv)\right|\leq C \nu h\norm{\bw}_2\trinorm{\bv},\label{eqn: suppest2} \\
& \left|\mathbf{s}(\Pi_h\bw,\bv)\right|\leq C\nu h\norm{\bw}_2\trinorm{\bv},\label{eqn: suppest4}
\end{alignat}
\end{subequations}
where the constant $C$ is independent of $h$.
\end{lemma}
\begin{proof}
The proof of the bound \eqref{eqn: suppest1} can be found in \cite{mu2015modified}, so we focus on showing \eqref{eqn: suppest2} and \eqref{eqn: suppest4} here.
The definition of the weak gradient and the properties of the projections $\Pi_h$ and $\Theta_h$ lead to
\begin{align*}
    |l_2(\bw,\bv)|&=\nu\left|\left(\nabla_w(\Pi_h\bw-\Theta_h\bw),\nabla_w\bv)\right)_{\Th}\right|\\
    &= \nu \left| \sum_{T\in \Th}\langle \avg{\Pi_h\bw}-\Theta_b\bw,\nabla_w\bv\cdot\bn\rangle_{\partial T}\right|\\
    &= \nu \left| \sum_{T\in \Th}\langle \avg{\Pi_h\bw-\bw},\nabla_w\bv\cdot\bn\rangle_{\partial T}\right|\\
    &\leq \nu \sum_{T\in\Th}\norm{h^{-1/2}_T\avg{\Pi_h\bw-\bw}}_{\partial T}\norm{h^{1/2}_T\nabla_w\bv}_{\partial T}\\
    &\leq C\nu h\norm{\bw}_2 \norm{\nabla_w\bv}_{0,\Th}
\end{align*}
The third identity holds true because $\bw\in [H^2(\Omega)]^d$ is continuous on $\partial T$, and the last inequality is obtained from the trace inequality \eqref{eqn: trace} and \eqref{eqn: Pih_err}.

For the stabilization term \eqref{eqn: suppest4}, it follows from the Cauchy-Schwarz inequality, \eqref{eqn: trace}, and \eqref{eqn: Pih_err} that
\begin{align*}
    |\mathbf{s}(\Pi_h\bw,\bv)|&=\nu\left|\langle h_e^{-1}\jump{\Pi_h\bw-\bw},\jump{\bv}\rangle_{\Eh}\right|\\
    &\leq C\nu\norm{h_e^{-1/2}\jump{\Pi_h\bw-\bw}}_{0,\Eh}\norm{h_e^{-1/2}\jump{\bv}}_{0,\Eh}\\
    &\leq C\nu h\norm{\bw}_2\trinorm{\bv}.
\end{align*}
\end{proof}

Consequently, we obtain the following error estimates.
\begin{theorem}\label{thm: erroresti}
Let $(\bu,p)\in [H_0^1(\Omega)\cap H^2(\Omega)]^d\times (L_0^2(\Omega)\cap H^1(\Omega))$ be the solution to \eqref{eqn: governing1}-\eqref{eqn: governing3}, and $(\buh,p_h)\in \Vh\times Q_h$ be the discrete solution from the \texttt{mEG} method. Then, we have the following error estimates
  \begin{align*}
      &\trinorm{\Pi_h\bu-\bu_h}\leq Ch\left(\norm{\bu}_2+\frac{1}{\nu}\norm{p}_1\right),\\
      &\norm{\Pz p-p_h}_0\leq Ch\left(\nu\norm{\bu}_2+\norm{p}_1\right).
  \end{align*}
\end{theorem}
\begin{proof}
First, we see the error equation \eqref{eqn: erreqn1}, for any $\bv\in\Vh$ and $q\in Q_h$,
\begin{equation*}
    \bb_w(\bv,\epsilon_h)=\ba_w(\be_h,\bv)-l_1(\bu,\bv)-l_2(\bu,\bv)-\mathbf{s}(\Pi_h\bu,\bv)-\bb_w(\bv,\xi_h).
\end{equation*}
Then, it follows from \eqref{eqn: contia}, \eqref{sys: suppest}, \eqref{eqn: contib}, and \eqref{eqn: P_err} that
\begin{align*}
      |\bb_w(\bv,\epsilon_h)|&\leq C\left(\nu\trinorm{\be_h}\trinorm{\bv}+\nu h\norm{\bu}_2\trinorm{\bv}+\norm{\xi_h}_0\trinorm{\bv}\right)\\
      &\leq C\left(\nu\trinorm{\be_h}+\nu h\norm{\bu}_2+h\norm{p}_1\right)\trinorm{\bv}.
  \end{align*}
The inf-sup condition \eqref{eqn: infsupb} implies that
\begin{equation}
    \norm{\epsilon_h}_0\leq C\left(\nu\trinorm{\be_h}+h(\nu\norm{\bu}_2+\norm{p}_1)\right).\label{eqn: interest}
\end{equation}
Moreover, by choosing $\bv=\be_h$ and $q=\epsilon_h$ in \eqref{sys: erreqn} and substituting \eqref{eqn: erreqn2} into \eqref{eqn: erreqn1}, we obtain
\begin{equation*}
    \ba_w(\be_h,\be_h)=-\bb_w(\bm{\chi}_h,\epsilon_h)+l_1(\bu,\be_h)+l_2(\bu,\be_h)+\mathbf{s}(\Pi_h\bu,\be_h)+\bb_w(\be_h,\xi_h).
\end{equation*}
Here, we show an upper bound for the term $\bb_w(\bm{\chi}_h,\epsilon_h)$.
Integration by parts and the trace identity \eqref{eqn: jump-avg} give 
\begin{align*}
    \bb_w(\bm{\chi}_h,\epsilon_h)&=\bb(\bm{\chi}_h,\epsilon_h)\\
    &=\left(\nabla\cdot\bm{\chi}_h,\epsilon_h\right)_{\Th}-\langle\jump{\bm{\chi}_h}\cdot\bn_e,\avg{\epsilon_h}\rangle_{\Eh}\\
    &=\sum_{T\in\Th}\langle \bm{\chi}_h\cdot\bn,\epsilon_h\rangle_{\partial T}-\langle\jump{\bm{\chi}_h}\cdot\bn_e,\avg{\epsilon_h}\rangle_{\Eh}\\
    &=\langle\avg{\bm{\chi}_h}\cdot\bn_e,\jump{\epsilon_h}\rangle_{\Eho}.
\end{align*}
Thus, it follows from the Cauchy-Schwarz inequality, \eqref{eqn: trace}, and \eqref{eqn: Pih_err} that
\begin{align}\label{intermediate result 2}
    |\bb_w(\bm{\chi}_h,\epsilon_h)|&\leq \norm{\avg{\bm{\chi}_h}}_{0,\Eh}\norm{\jump{\epsilon_h}}_{0,\Eh}
    \leq Ch\norm{\bu}_2\norm{\epsilon_h}_0.
\end{align}
Hence, by \eqref{eqn: coera}, \eqref{sys: suppest}, \eqref{eqn: contib}, \eqref{eqn: P_err}, \eqref{eqn: interest}, and \eqref{intermediate result 2}, we have
\begin{equation*}
    \nu\trinorm{\be_h}^2\leq C\left(\nu h\norm{\bu}_2\trinorm{\be_h}+h\norm{p}_1\trinorm{\be_h}+\nu h^2\norm{\bu}_2^2+h^2\norm{\bu}_2\norm{p}_1\right).
\end{equation*}
We also apply the Young's inequality with a positive constant $\kappa$ satisfying $\kappa<1/C$,
\begin{align*}
    &\nu h\norm{\bu}_2\trinorm{\be_h}\leq \nu\left(\frac{h^2}{2\kappa}\norm{\bu}_2^2+\frac{\kappa}{2}\trinorm{\be_h}^2\right),\\
    &h\norm{p}_1\trinorm{\be_h}\leq  \left(\frac{h^2}{2\nu\kappa}\norm{p}_1^2+\frac{\nu\kappa}{2}\trinorm{\be_h}^2\right),\\
    &h^2\norm{\bu}_2\norm{p}_1\leq\left(\frac{\nu h^2}{2}\norm{\bu}^2_2+\frac{h^2}{2\nu}\norm{p}_1^2\right).
\end{align*}
We finally obtain
\begin{equation*}
\nu\trinorm{\be_h}^2\leq C\left(\nu h^2\norm{\bu}_2^2+\frac{h^2}{\nu}\norm{p}_1^2\right),
\end{equation*}
which implies that
\begin{equation*}
    \trinorm{\be_h}\leq Ch\left(\norm{\bu}_2+\frac{1}{\nu}\norm{p}_1\right).
\end{equation*}
In addition, together with this velocity error estimate, 
the estimate \eqref{eqn: interest} implies
\begin{equation*}
    \quad\norm{\epsilon_h}_0\leq Ch\left(\nu\norm{\bu}_2+\norm{p}_1\right).
\end{equation*}
\end{proof}

Finally, we present the total error estimates showing the optimal rates of convergence in both velocity and pressure.

\begin{theorem}
Under the same assumption of Theorem~\ref{thm: erroresti}, we have the following error estimates
  \begin{align*}
      &\trinorm{\bu-\bu_h}\leq Ch\left(\norm{\bu}_2+\frac{1}{\nu}\norm{p}_1\right),\\
      &\norm{p-p_h}_0\leq Ch\left(\nu\norm{\bu}_2+\norm{p}_1\right).
  \end{align*}
\end{theorem}
\begin{proof}
The estimates in this theorem are readily proved by the triangle inequality, the interpolation error estimates \eqref{eqn: Pih_energy_err} and \eqref{eqn: P_err}, and the norm equivalence \eqref{eqn: normeq}.
\end{proof}

%%%%%%%%%%%%%%%%%%%%%%%%%%%%%%%%%%
% 5. A Pressure-Robust Modified Enriched Galerkin Method
%%%%%%%%%%%%%%%%%%%%%%%%%%%%%%%%%%

\section{A Pressure-Robust Modified Enriched Galerkin Method}\label{sec:PRmEG}

In this section, we derive a pressure-robust scheme associated with the \texttt{mEG} method (Algorithm~\ref{alg:mEG}) by applying the velocity reconstruction operator \cite{HuLeeMuYi} to the load vector on the right hand side.
The operator $\cR: \Vh \to  \mathcal{B}DM_1(\Th)\subset H(\text{div},\Omega)$ is defined by
\begin{subequations}\label{sys: BDM}
\begin{alignat}{2}
\int_e (\cR \bv) \cdot\bn_e  p_1\  ds & = \int_e \avg{\bv}\cdot\bn_e p_1 \ ds,
 && \quad \forall p_1 \in P_1(e), \ \forall e \in \Eho,  \\
\int_e (\cR \bv) \cdot\bn_e  p_1\  ds & = 0,  && \quad \forall p_1 \in P_1(e), \ \forall e \in \Ehb,
\end{alignat}
\end{subequations}
when $\mathcal{B}DM_1(\Th)$ denotes the Brezzi-Douglas-Marini space of index 1 on $\Th$.

\begin{algorithm}[H]
\caption{Pressure-robust modified enriched Galerkin (\texttt{PR-mEG}) method}
\label{alg:PR-mEG}
Find $( \bu_h, \ph) \in \Vh \times Q_h $ such that
\begin{subequations}\label{sys:PR-mEG}
\begin{alignat}{2}
\ba_w(\buh,\bv)  - \bb_w(\bv, \ph) &= (\bbf, \cR \bv)_{\Th}, &&\quad \forall \bv \in\Vh, \label{eqn: prmeg1}\\
\bb_w(\buh,q) &= 0, &&\quad \forall q\in Q_h,  \label{eqn: prmeg2}
\end{alignat}
\end{subequations}
where $\ba_w(\cdot, \cdot)$ and $\bb_w(\cdot, \cdot)$ are the same as \eqref{eqn: biaw} and \eqref{eqn: bibw}, respectively.
\end{algorithm}

\begin{remark}
The \texttt{mEG} method in Algorithm~\ref{alg:mEG} and \texttt{PR-mEG} method in Algorithm~\ref{alg:PR-mEG} have the same formulation on left hand side that consists of $\ba_w(\cdot,\cdot)$ and $\bb_w(\cdot,\cdot)$.
The only difference is that a reconstructed test function is applied to the load vector on the right hand side.
This implies that the well-posedness of the \texttt{PR-mEG} method is guaranteed by that of the \texttt{mEG} method, and moreover, both of the \texttt{mEG} and \texttt{PR-mEG} methods produce
the same stiffness matrix.
\end{remark}

The error equations corresponding to the \texttt{PR-mEG} method are derived in the following lemma using the same error functions in \eqref{eqn: errfunctions}.

\begin{lemma}
For any $\bv\in\Vh$ and $q\in Q_h$, we have
\begin{subequations}\label{erreqnpr}
\begin{alignat}{2}
\ba_w(\be_h,\bv)-\bb_w(\bv,\epsilon_h)&=l_1(\bu,\bv)+l_2(\bu,\bv)+l_3(\bu,\bv)+\mathbf{s}(\Pi_h\bu,\bv),\label{eqn: erreqnpr1}\\
\bb_w(\be_h,q)&=-\bb_w(\bm{\chi}_h,q),\label{eqn: erreqnpr2}
\end{alignat}
\end{subequations}
where $l_1(\cdot,\cdot)$, $l_2(\cdot,\cdot)$, and $\mathbf{s}(\cdot,\cdot)$ are defined in Lemma~\ref{lemma: erreqn}, and another supplemental bilinear form is defined by
\begin{equation*}
    l_3(\bu,\bv)=-\nu\left(\Delta\bu,\bv-\cR\bv\right)_{\Th}
\end{equation*}
\end{lemma}
\begin{proof}
First of all, we obtain the following identities,
\begin{equation*}
    \left(\nabla p,\cR\bv\right)_{\Th} =-\bb(\bv,\Pz p)= -\bb_w(\bv,\Pz p)
\end{equation*}
because $\cR\bv\cdot\bn$ is continuous on $\partial T$ and $\nabla\cdot\cR\bv$ is constant in $T$. (See \cite{HuLeeMuYi} for details.)
Moreover, we have
\begin{equation*}
    \left(-\Delta\bu,\cR\bv\right)_{\Th}=\left(-\Delta\bu,\bv\right)_{\Th} + \left(\Delta\bu,\bv-\cR\bv\right)_{\Th}.
\end{equation*}
Then, it follows from \eqref{eqn: governing1} and the error equations in Lemma~\ref{lemma: erreqn} that
\begin{equation*}
    \ba_w(\Pi_h\bu,\bv)-\bb_w(\bv,\Pz p)=(\bbf,\cR \bv)_{\Th}+l_1(\bu,\bv)+l_2(\bu,\bv)+l_3(\bu,\bv)+\mathbf{s}(\Pi_h\bu,\bv).
\end{equation*}
By subtracting \eqref{eqn: prmeg1} from this equation, we arrive at the equation \eqref{eqn: erreqnpr1}.
The equation \eqref{eqn: erreqnpr2} is simply derived in the same way as Lemma~\ref{lemma: erreqn}.
\end{proof}

Consequently, the following theorem theoretically shows pressure-robustness of the \texttt{PR-mEG} method.

\begin{theorem}
Let $(\bu,p)\in [H_0^1(\Omega)\cap H^2(\Omega)]^d\times (L_0^2(\Omega)\cap H^1(\Omega))$ be the solution to \eqref{eqn: governing1}-\eqref{eqn: governing3}, and $(\buh,p_h)\in \Vh\times Q_h$ be the discrete solution from the \texttt{PR-mEG} method. Then, we have the following error estimates
  \begin{equation*}
      \trinorm{\Pi_h\bu-\bu_h}\leq Ch\norm{\bu}_2,\quad\norm{\mathcal{P}_0p-p_h}_0\leq C\nu h\norm{\bu}_2.
  \end{equation*}
Therefore, the total error estimates are as follows:
\begin{equation*}
      \trinorm{\bu-\bu_h}\leq Ch\norm{\bu}_2,\quad\norm{p-p_h}_0\leq Ch \left(\nu\norm{\bu}_2+\norm{p}_1\right).
  \end{equation*}
  
\end{theorem}

\begin{proof}
To begin with, we observe the error equation \eqref{eqn: erreqnpr1},
\begin{equation*}
    \bb_w(\bv,\epsilon_h)=\ba_w(\be_h,\bv)-l_1(\bu,\bv)-l_2(\bu,\bv)-l_3(\bu,\bv)-\mathbf{s}(\Pi_h\bu,\bv).
\end{equation*}
Here, the bilinear form $l_3(\bu,\bv)$ is bounded using the Cauchy-Schwarz inequality,
\begin{align*}
    |l_3(\bu,\bv)|\leq \nu\norm{\Delta\bu}_0\norm{\bv-\cR\bv}_0\leq \nu \norm{\bu}_2\norm{\bv-\cR\bv}_0.
\end{align*}
It also follows from the estimate $\norm{\bv-\cR\bv}_0$ in \cite{HuLeeMuYi} and the norm equivalence \eqref{eqn: normeq} that
\begin{equation*}
   \norm{\bv-\cR\bv}_0\leq Ch\trinorm{\bv},
\end{equation*}
so we arrive at
\begin{equation}
    |l_3(\bu,\bv)|\leq C\nu h\norm{\bu}_2\trinorm{\bv}.\label{eqn: suppest3}
\end{equation}
Thus,
from \eqref{eqn: contia}, \eqref{sys: suppest}, and \eqref{eqn: suppest3}, we obtain
\begin{equation*}
    |\bb_w(\bv,\epsilon_h)|\leq C\left(\nu\trinorm{\be_h}+\nu h\norm{\bu}_2\right)\trinorm{\bv}.
\end{equation*}
Hence, the inf-sup condition \eqref{eqn: infsupb} leads to
\begin{equation}
    \norm{\epsilon_h}_0\leq C\nu\left(\trinorm{\be_h}+h\norm{\bu}_2\right).\label{eqn: intresult}
\end{equation}
Similar to the proof of Theorem~\ref{thm: erroresti}, choosing $\bv=\be_h$ and $q=\epsilon_h$ yields that
\begin{equation*}
    \ba_w(\be_h,\be_h)=-\bb_w(\bm{\chi}_h,\epsilon_h)+l_1(\bu,\be_h)+l_2(\bu,\be_h)+l_3(\bu,\be_h)+\mathbf{s}(\Pi_h\bu,\be_h).
\end{equation*}
From \eqref{intermediate result 2} and \eqref{eqn: intresult}, we get the following intermediate result, 
\begin{equation*}
    |\bb_w(\bm{\chi}_h,\epsilon_h)|
    \leq Ch\norm{\bu}_2\norm{\epsilon_h}_0
    \leq C\nu h\norm{\bu}_2\left(\trinorm{\be_h}+h\norm{\bu}_2\right).
\end{equation*}
Therefore, it follows from \eqref{eqn: coera}, \eqref{sys: suppest}, and \eqref{eqn: suppest3} that
\begin{equation*}
    \nu\trinorm{\be_h}^2\leq C\nu\left( h\norm{\bu}_2\trinorm{\be_h}+ h^2\norm{\bu}_2^2\right).
\end{equation*}
The Young's inequality gives
\begin{equation*}
    h\norm{\bu}_2\trinorm{\be_h}\leq\frac{h^2}{2\kappa}\norm{\bu}_2^2+\frac{\kappa}{2}\trinorm{\be_h}^2,
\end{equation*}
so choosing a proper $\kappa$ implies
\begin{equation*}
    \nu\trinorm{\be_h}^2\leq C\nu h^2\norm{\bu}_2^2.
\end{equation*}
Therefore, together with \eqref{eqn: intresult}, we obtain
\begin{equation*}
    \trinorm{\be_h}\leq Ch\norm{\bu}_2,\quad \norm{\epsilon_h}_0\leq C\nu h\norm{\bu}_2.
\end{equation*}
\end{proof}

%%%%%%%%%%%%%%%%%%%%%%%%%%%%%%%%%%
% 6. Numerical Experiments
%%%%%%%%%%%%%%%%%%%%%%%%%%%%%%%%%%

\section{Numerical Experiments}
\label{sec:nume}

In this section, we present numerical experiments validating our theoretical results with two- and three-dimensional examples.
The numerical experiments are implemented by authors' codes developed based on iFEM \cite{CHE09}.
The numerical methods mentioned in this paper and their discrete solutions are denoted as follows:
\begin{itemize}
    \item $(\bu_h^{\texttt{EG}},p_h^{\texttt{EG}})$: Solution by the \texttt{EG} method \cite{YiEtAl22-Stokes} in Algorithm~\ref{alg:EG}.
    \item $(\bu_h^{\texttt{mEG}},p_h^{\texttt{mEG}})$: Solution by the \texttt{mEG} method in Algorithm~\ref{alg:mEG}.
    \item $(\bu_h^{\texttt{PR}},p_h^{\texttt{PR}})$: Solution by the \texttt{PR-mEG} method in Algorithm~\ref{alg:PR-mEG}.
\end{itemize}
We compare the penalty terms in the \texttt{EG} and \texttt{mEG} methods,
\begin{align}
    &\texttt{EG}: \text{Penalty term of }\ba(\bu_h^{\texttt{EG}},\bv)\quad\rightarrow\quad\nu\rho \langle h_e^{-1}\jump{\bu_h^{\texttt{EG}}},
\jump{\bv}\rangle_{\Eh},\label{penaltyEG}\\
&\texttt{mEG}:\text{Penalty term of }\ba_w(\bu_h^{\texttt{mEG}},\bv)\quad\rightarrow\quad\nu\rho_{\texttt{m}}\langle h_e^{-1}\jump{\bu_h^{\texttt{mEG}}},
\jump{\bv}\rangle_{\Eh}\label{penaltymEG},
\end{align}
where $\rho_{\texttt{m}}$ is a hypothetical \texttt{mEG} penalty parameter for comparison and the \texttt{mEG} method uses the fixed parameter $\rho_{\texttt{m}}=1$.
While a sufficiently large penalty parameter $\rho$ is required for the \texttt{EG} method, our \texttt{mEG} method ($\rho_\texttt{m}=1$) is a parameter-free EG method under the same finite dimensional velocity and pressure spaces.
We recall the error estimates for the \texttt{mEG} method in Section \ref{sec:wellerror}:
\begin{subequations}\label{sys: errboundst}
\begin{alignat}{1}
 \trinorm{\Pi_h\bu-\bu_h^{\texttt{mEG}}}\lesssim h \left( \|\mathbf{u}\|_2+ \nu^{-1}\|p\|_1 \right ), \quad 
&\|\mathcal{P}_0p-p_h^{\texttt{mEG}}\|_0\lesssim  h \left( \nu \|\mathbf{u}\|_2+ \|p\|_1 \right ), \label{eqn: errboundstu} \\
 \trinorm{\bu-\bu_h^{\texttt{mEG}}}\lesssim h \left( \|\mathbf{u}\|_2+ \nu^{-1}\|p\|_1 \right ), \quad 
&\|p-p_h^{\texttt{mEG}}\|_0 \lesssim  h \left ( \nu\|\mathbf{u}\|_2+ \|p\|_1\right ),\label{eqn: errboundstp}
\end{alignat}
\end{subequations}
which means the same rates of convergence as the \texttt{EG} method.
Moreover, we developed a pressure-robust numerical scheme corresponding to the \texttt{mEG} method, and the error estimates for the \texttt{PR-mEG} method proved in Section \ref{sec:PRmEG} are as follows:
\begin{subequations}\label{sys: errboundpr}
\begin{alignat}{1}
 \trinorm{\Pi_h\bu-\bu_h^{\texttt{PR}}}\lesssim h \|\mathbf{u}\|_2 , \quad 
&\|\mathcal{P}_0p-p_h^{\texttt{PR}}\|_0\lesssim   \nu h \|\mathbf{u}\|_2, \label{eqn: errboundpru} \\
 \trinorm{\bu-\bu_h^{\texttt{PR}}}\lesssim h \|\mathbf{u}\|_2, \quad 
&\|p-p_h^{\texttt{PR}}\|_0 \lesssim  h \left ( \nu\|\mathbf{u}\|_2+ \|p\|_1\right ).\label{eqn: errboundprp}
\end{alignat}
\end{subequations}
In two- and three-dimensional examples, we demonstrate the well-posedness and optimal rates of convergence for the \texttt{mEG} method.
By checking the behaviors of the errors with decreasing viscosity $\nu$, we confirm the error estimates of the \texttt{PR-mEG} method in \eqref{sys: errboundpr}, which means more accurate numerical solutions than the \texttt{mEG} method in the case of small viscosity $\nu\ll1$.

%%%%%%%%%% 6.1 %%%%%%%%%%%

\subsection{Two dimensional examples}

Let the computational domain be $\Omega=(0,1)\times (0,1)$. The velocity field and pressure are chosen as
\begin{equation}
    \mathbf{u}
    = \left(\begin{array}{c}
    10x^2(x-1)^2y(y-1)(2y-1) \\
    -10x(x-1)(2x-1)y^2(y-1)^2
    \end{array}\right),
    \quad
    p = 10(2x-1)(2y-1).
    \label{eqn: example1}
\end{equation}
Then, the body force $\mathbf{f}$ is obtained from the Stokes equations in \eqref{sys: governing}, and 
the homogeneous boundary condition for velocity is considered.

% 6.1.1

\subsubsection{Parameter-free test}
We check the errors and condition numbers of the stiffness matrices for the \texttt{EG} and \texttt{mEG} methods
with different penalty parameters.
To see how the penalty parameters affect the performance of the two methods, we apply the penalty terms \eqref{penaltyEG}-\eqref{penaltymEG} and change $\rho$ and $\rho_\texttt{m}$ from 0.1 to 5.
In the test, we choose the uniform triangular mesh with $h=1/16$ and the viscosity $\nu=1$.
\begin{figure}[h!]
    \centering
    \includegraphics[width=.3\textwidth]{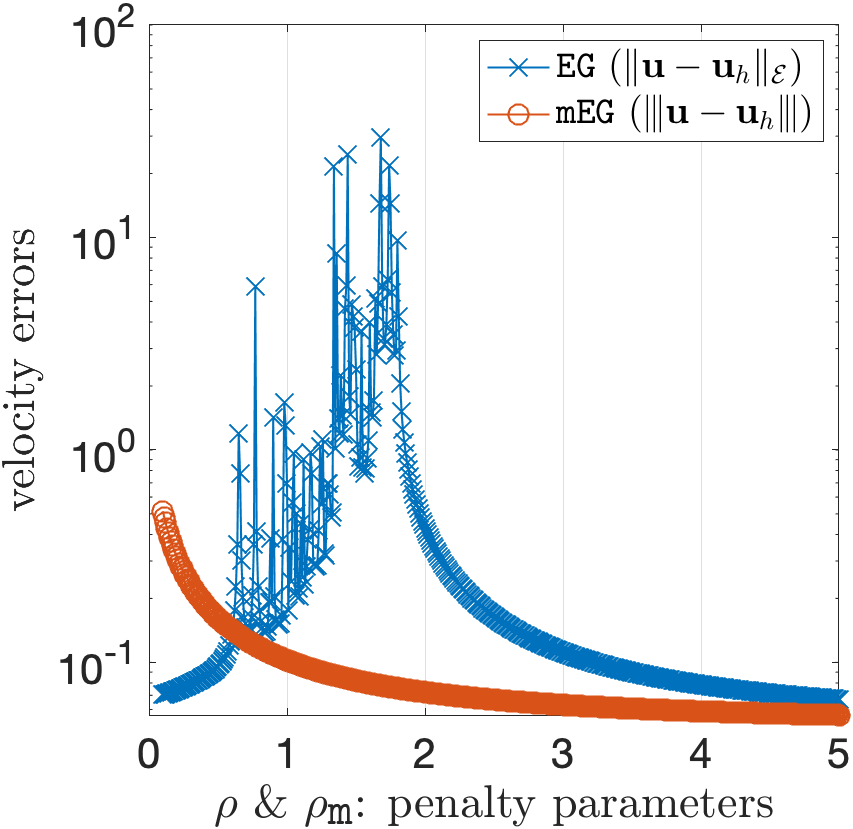}
    \hskip 5pt
    \includegraphics[width=.3\textwidth]{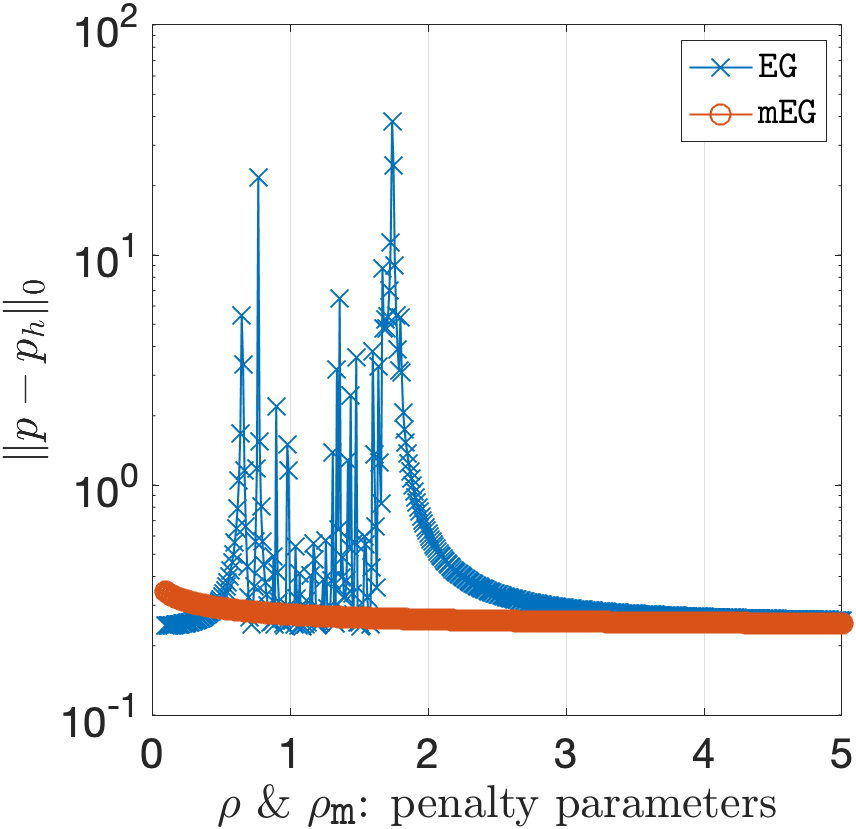}
    \hskip 5pt
    \includegraphics[width=.3\textwidth]{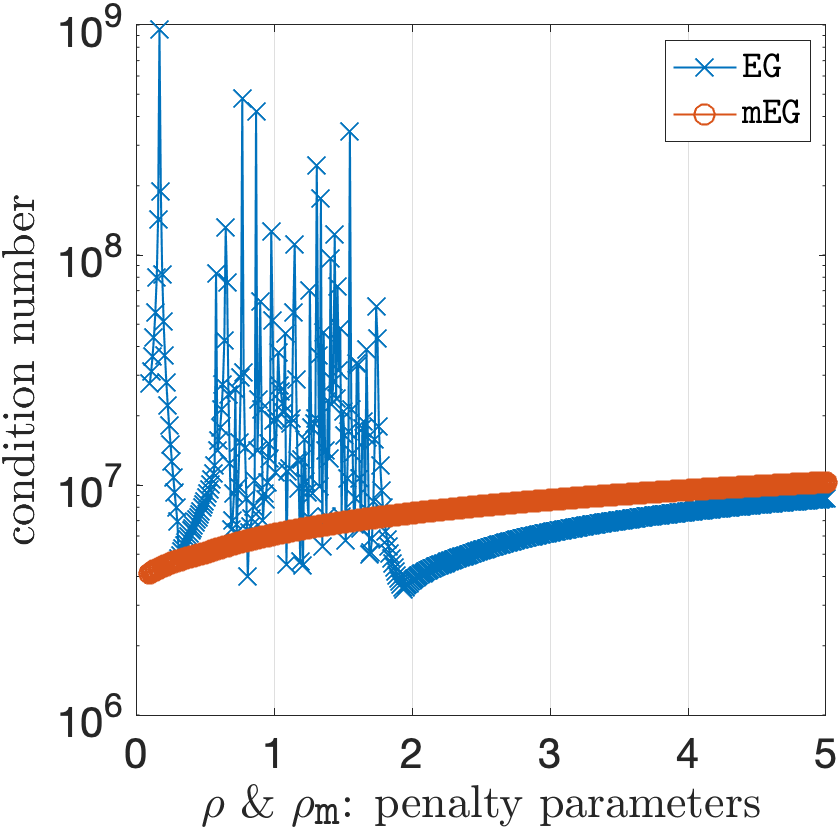}
    \caption{Errors and condition numbers of \texttt{EG} and \texttt{mEG} for $0.1\leq\rho,\rho_\texttt{m}\leq 5$ ($\nu=1$, $h=1/16$).}
    \label{figure: para-errors}
\end{figure}
Figure \ref{figure: para-errors} shows that
the \texttt{EG} method seems to yield unstable errors and condition numbers with penalty parameters less than 2, which implies the need of a sufficiently large parameter for the stability.
On the other hand, the \texttt{mEG} method shows stable behaviors in the errors and condition numbers for any positive parameter $\rho_\texttt{m}$.

\begin{table}[h!]
    \centering
    \begin{tabular}{|c||c|c||c|c|}
    \hline
        &  \multicolumn{2}{c||}{\texttt{EG} ($\rho=1$)
        } & 
         \multicolumn{2}{c|}{\texttt{mEG} ($\rho_\mathtt{m}=1$)} \\
    \cline{2-5}   
       $h$  & { $\|\mathbf{u}-\mathbf{u}_h^{\texttt{EG}}\|_\mathcal{E}$} & {Rate} &
       { $\trinorm{\mathbf{u}-\mathbf{u}_h^{\texttt{mEG}}}$} & {Rate} \\ 
       \hline
       $1/8$ & 7.394e-1 & - &  2.749e-1 & - \\
       \hline
       $1/16$  & 6.931e-1 & 0.09 & 1.024e-1 & 1.42  \\
       \hline
       $1/32$   & 2.440e-1 & 1.51 & 3.940e-2 & 1.38  \\
       \hline
       $1/64$  & 9.052e-2 & 1.43 & 1.606e-2 & 1.29  \\
       \hline
       \hline
        $h$ & {\small$\|p-p_h^{\texttt{EG}}\|_0$} & {\small Rate} &
       {\small$\|p-p_h^{\texttt{mEG}}\|_0$} & {\small Rate}  \\ 
       \hline
       $1/8$ & 9.299e-1 & - &  5.815e-1 & -  \\
       \hline
       $1/16$  & 2.897e-1 & 1.68 & 2.733e-1 & 1.09 \\
       \hline
       $1/32$   & 2.319e-1 & 0.32 & 1.322e-1 & 1.05  \\
       \hline
       $1/64$  & 2.664e-1 & -0.20 & 6.498e-2 & 1.02 \\
       \hline
    \end{tabular}
    \caption{A mesh refinement study for \texttt{EG} and \texttt{mEG} with varying mesh size $h$ and $\nu=1$.}
    \label{table: para-errors}
\end{table}
We also perform a mesh refinement study for the \texttt{EG} and \texttt{mEG} methods when $\rho=\rho_\texttt{m}=1$.
In Table \ref{table: para-errors}, the errors of the \texttt{EG} method fail to converge due to the insufficiently large penalty parameter.
However, the \texttt{mEG} method produces the velocity and pressure errors that indicate at least the first-order convergence.

\begin{figure}[h!]
\centering
    \texttt{EG} ($\rho=1$): $u_1$, $u_2$, and $p$ from left to right\\
    \includegraphics[width=.3\textwidth]{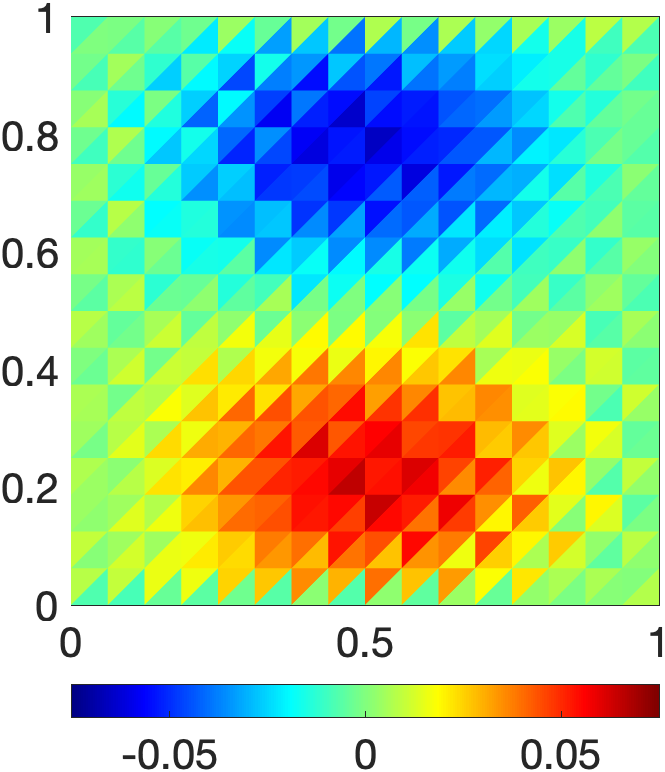}
    \hskip 10pt
    \includegraphics[width=.3\textwidth]{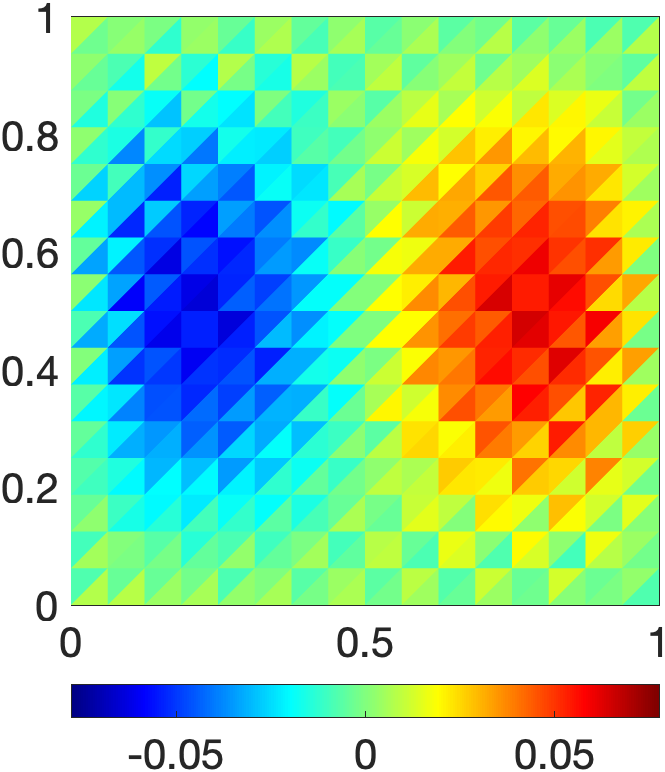}
    \hskip 10pt
    \includegraphics[width=.3\textwidth]{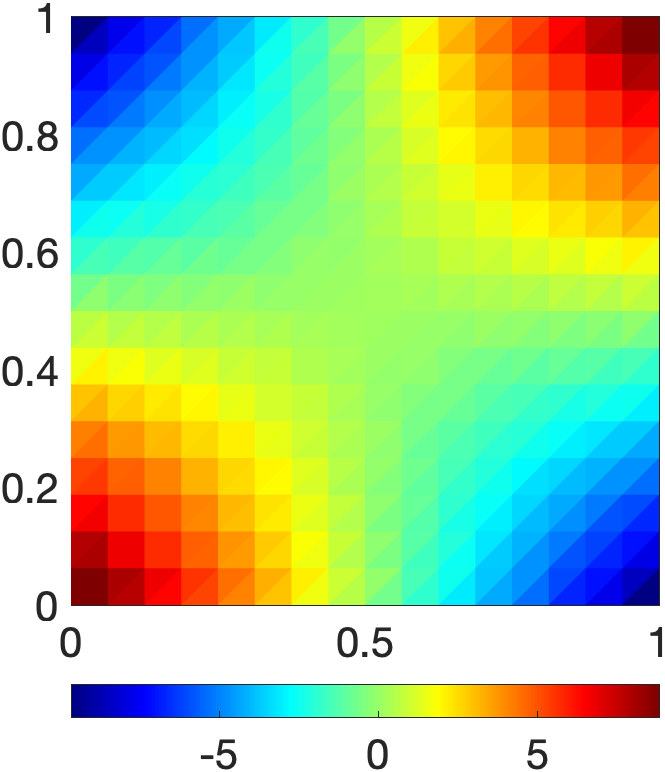}\\
    \vskip 10pt
    \texttt{mEG} ($\rho_\texttt{m}=1$): $u_1$, $u_2$, and $p$ from left to right\\
    \includegraphics[width=.3\textwidth]{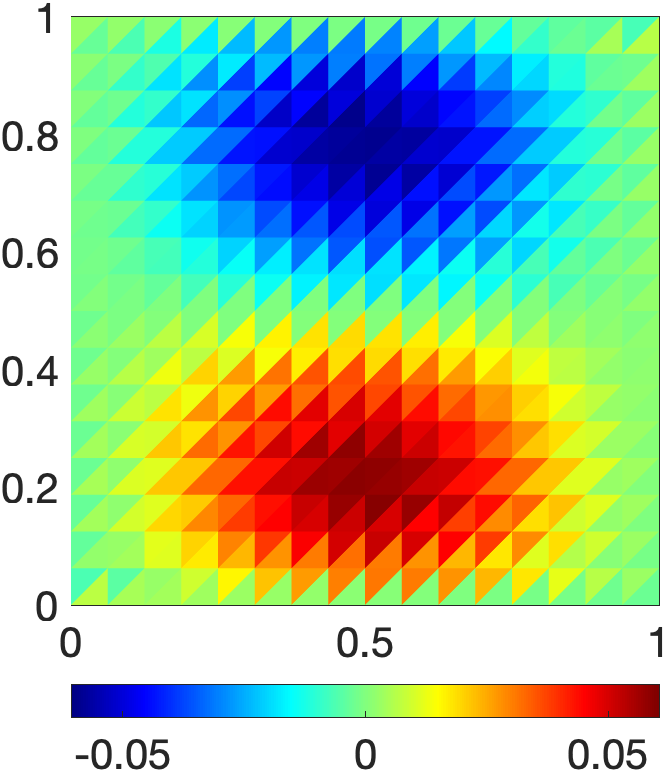}
    \hskip 10pt
    \includegraphics[width=.3\textwidth]{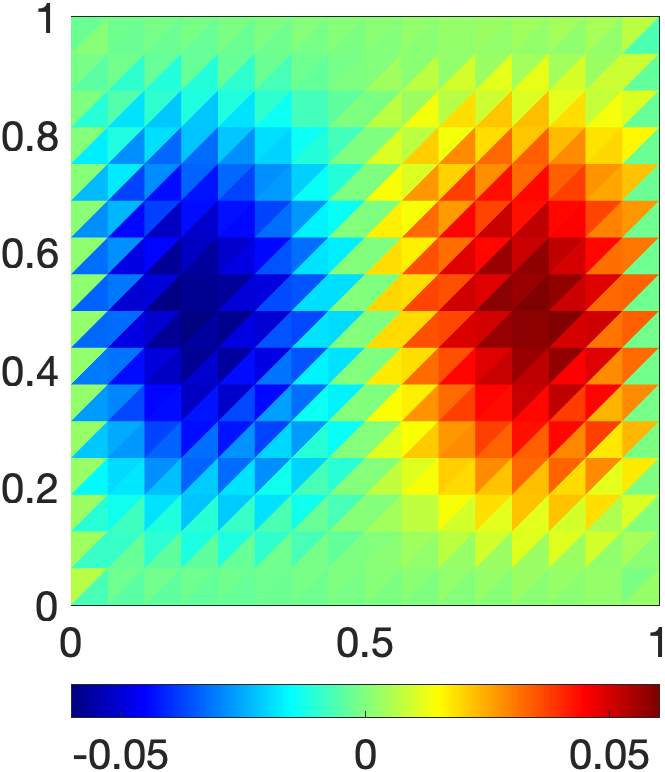}
    \hskip 10pt
    \includegraphics[width=.3\textwidth]{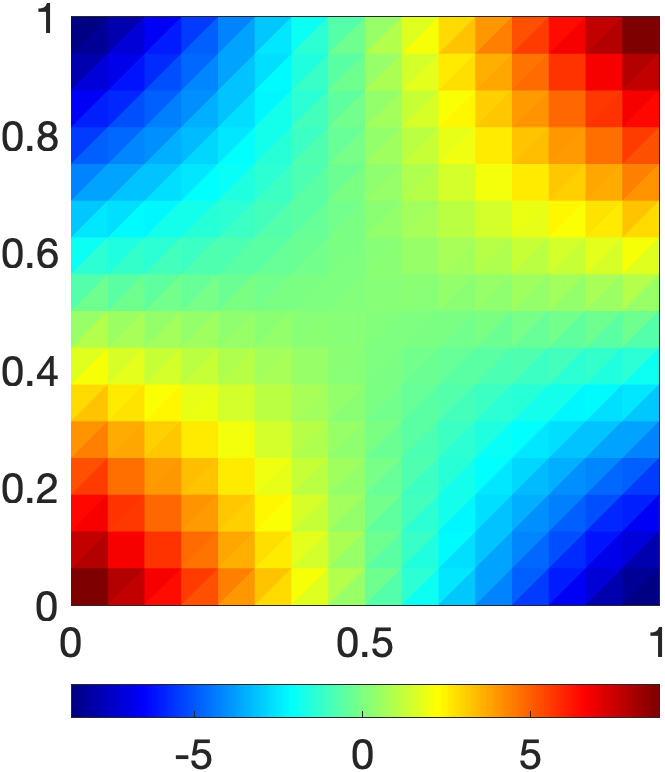}
    \caption{Comparison of the numerical solutions with $h=1/16$ and $\nu=1$.}
    \label{figure: compare numesol}
\end{figure}

Moreover, we compare the numerical solutions of the \texttt{EG} and \texttt{mEG} methods when $h=1/16$, $\nu=1$, and $\rho=\rho_{\texttt{m}}=1$.
In Figure \ref{figure: compare numesol}, the numerical velocity of the \texttt{EG} method roughly captures the vortex flow pattern, but some relatively large jumps appear throughout the numerical velocity solution.
The \texttt{mEG} method, however, produces more stable numerical velocity that well captures the pattern compared to the \texttt{EG} method.

\begin{figure}[h!]
\centering
    \texttt{Perturbed mesh}: its mesh quality and velocity errors with respect to $\rho$ and $\rho_\texttt{m}$\\
    \includegraphics[width=.26\textwidth]{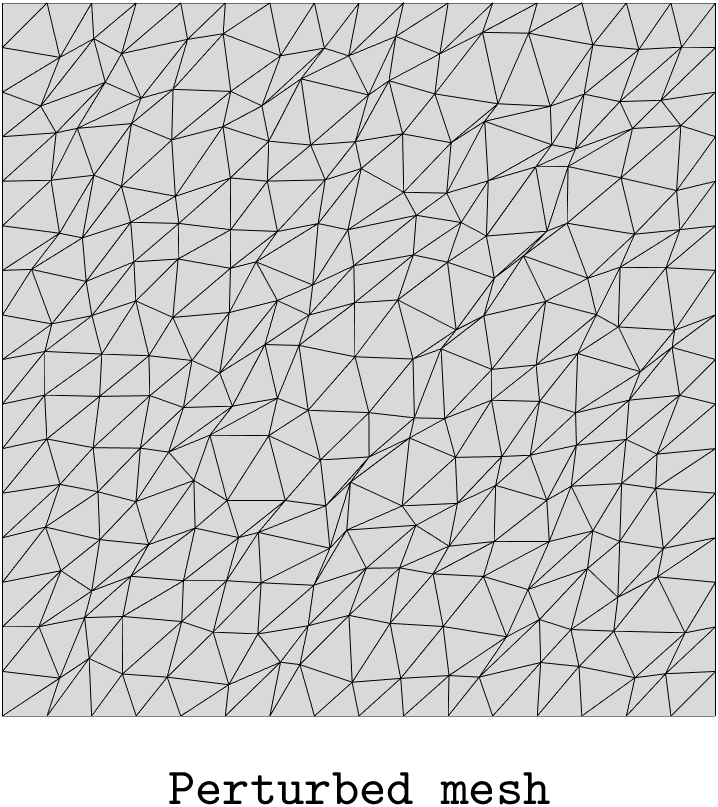}
    \hskip 5pt
    \includegraphics[width=.3\textwidth]{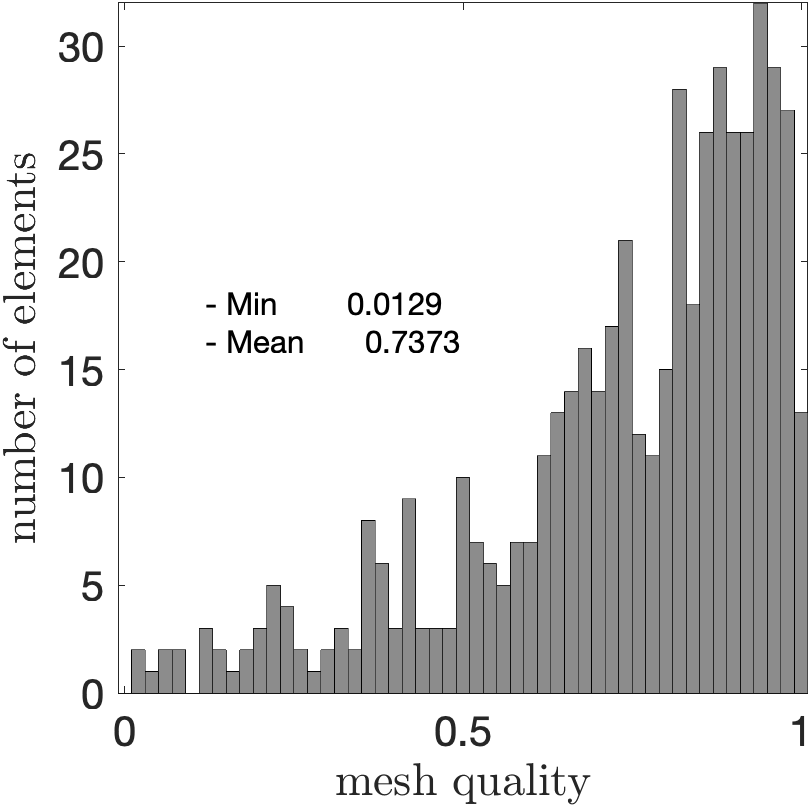}
    \hskip 5pt
    \includegraphics[width=.31\textwidth]{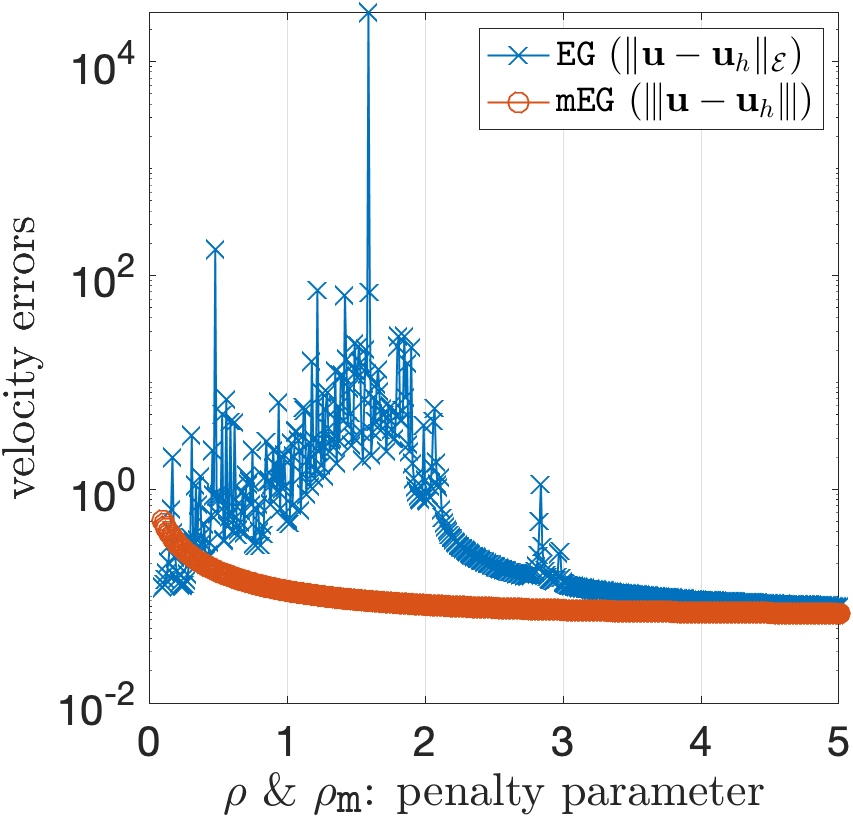}\\
    \vskip 10pt
    \texttt{Square with hole}: its mesh quality and velocity errors with respect to $\rho$ and $\rho_\texttt{m}$\\
    \includegraphics[width=.26\textwidth]{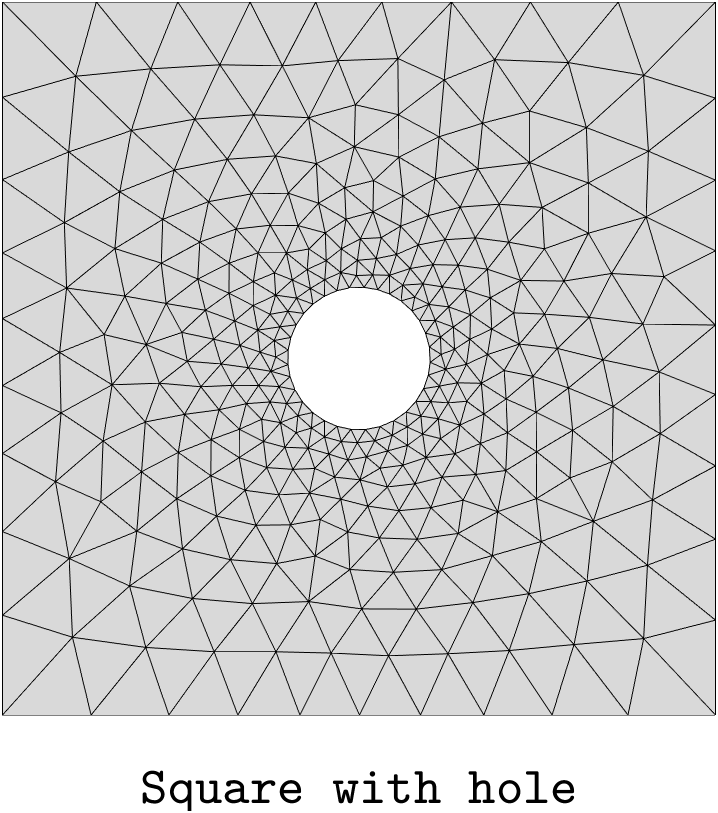}
    \hskip 5pt
    \includegraphics[width=.3\textwidth]{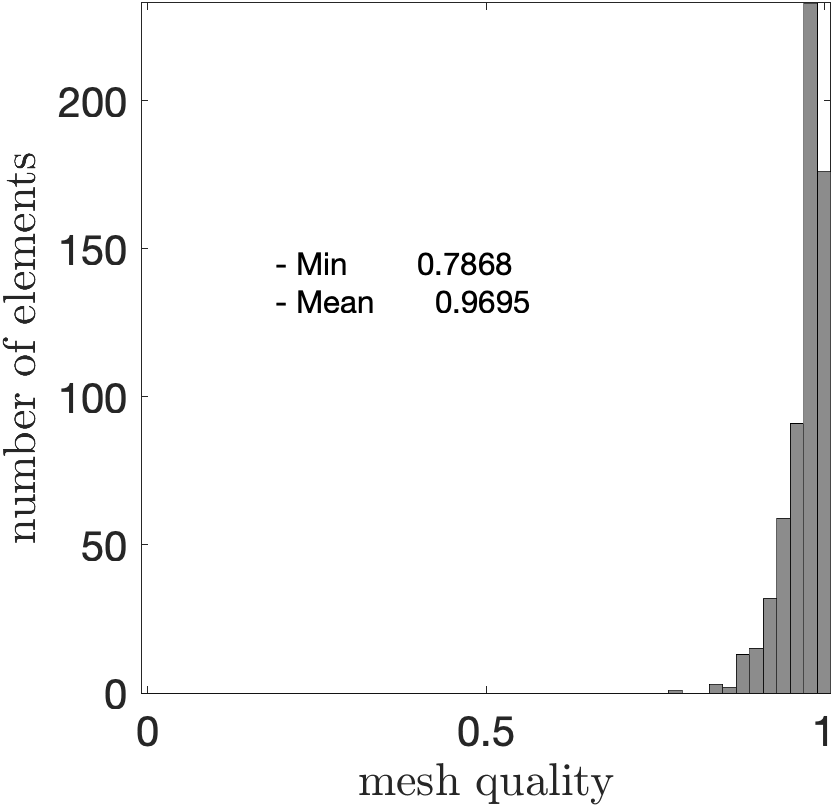}
    \hskip 5pt
    \includegraphics[width=.31\textwidth]{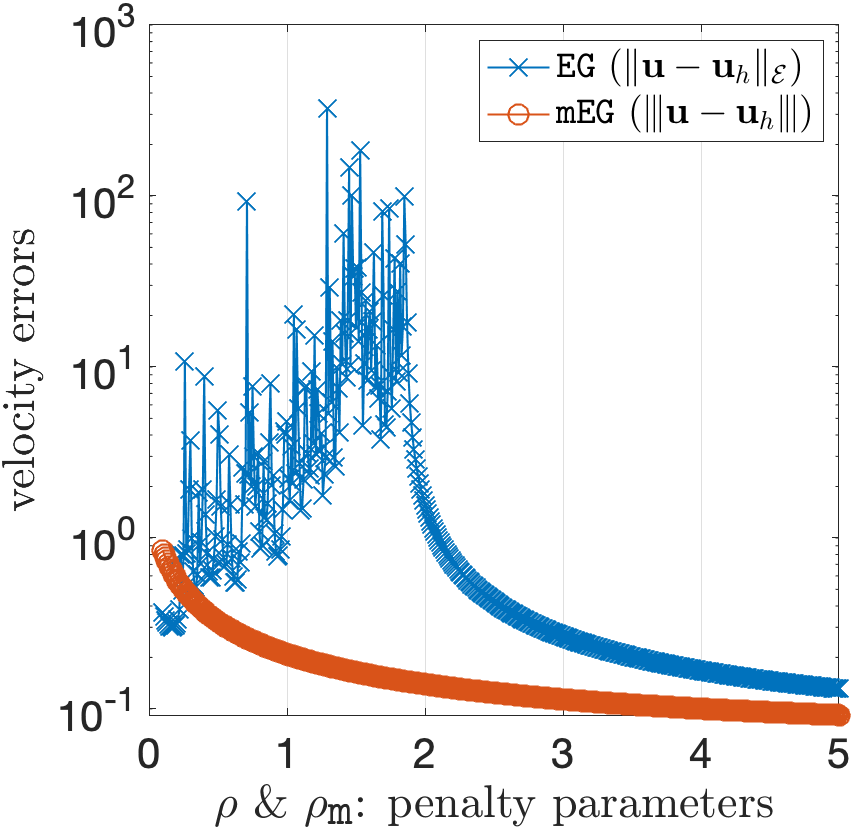}\\
    \vskip 10pt
    \texttt{L-shape}: its mesh quality and velocity errors with respect to $\rho$ and $\rho_\texttt{m}$\\
    \includegraphics[width=.26\textwidth]{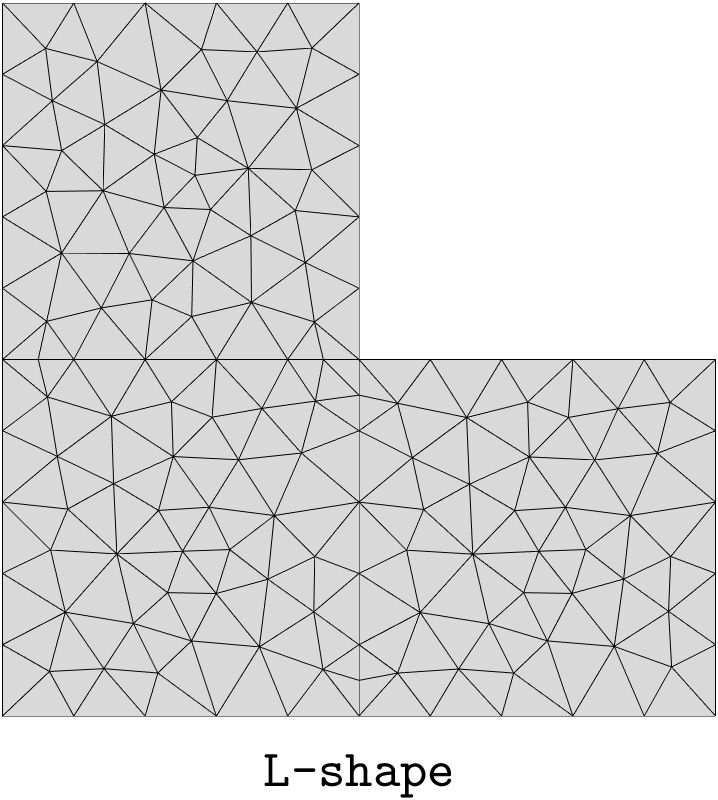}
    \hskip 5pt
    \includegraphics[width=.3\textwidth]{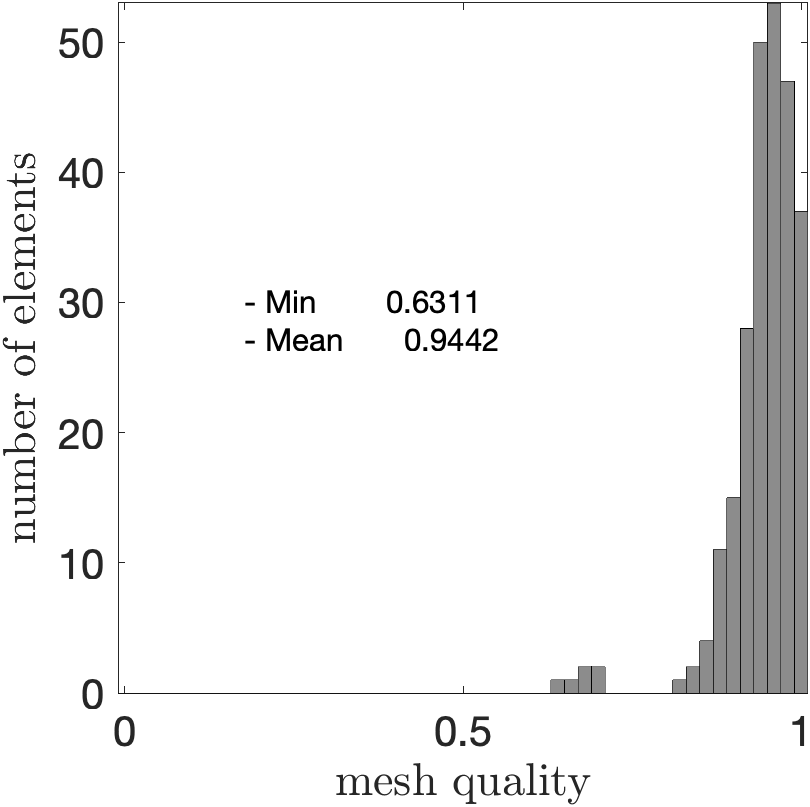}
    \hskip 5pt
    \includegraphics[width=.31\textwidth]{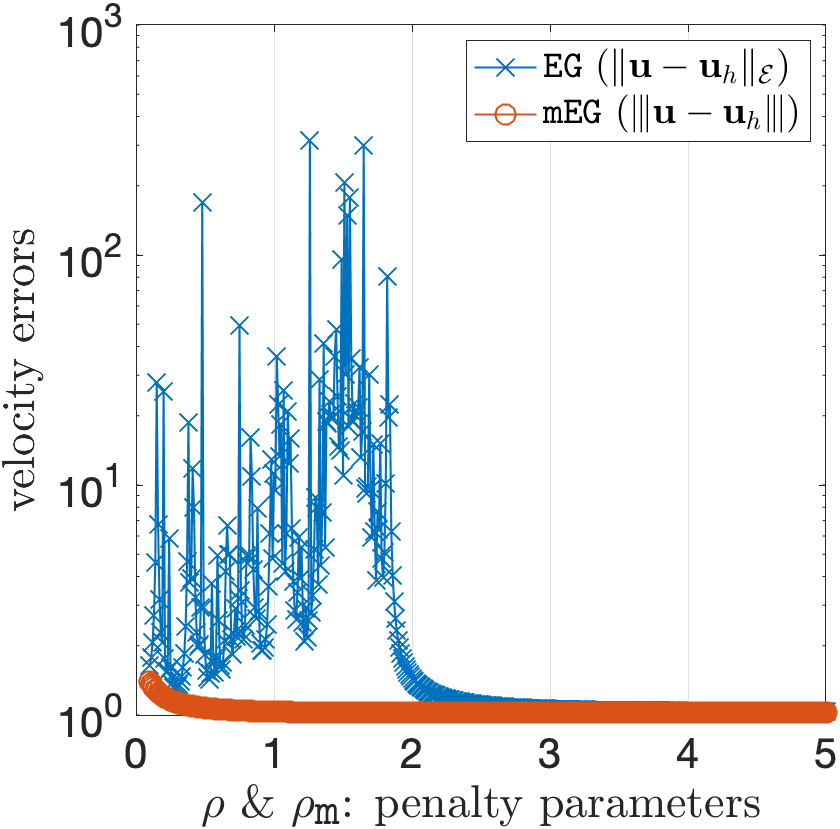}
    \caption{Comparison of \texttt{EG} and \texttt{mEG} for $0.1\leq\rho,\rho_\texttt{m}\leq 5$ on various meshes ($\nu=1$).}
    \label{figure: para-meshtype}
\end{figure}
We also conduct the parameter-free test on various meshes presented in Figure \ref{figure: para-meshtype}:
\begin{itemize}
    \item \texttt{Perturbed mesh}: The uniform triangular mesh is randomly perturbed, so some very sharp triangles are generated. The velocity field and pressure in \eqref{eqn: example1} are considered. The homogeneous boundary condition is preserved. 
    \item \texttt{Square with hole}: The computational domain is the unit square with a hole in the middle.
    An adaptive mesh is generated with triangles of different sizes.
    The velocity field and pressure in \eqref{eqn: example1} are considered. 
    The homogeneous boundary condition is preserved on the outer square, but a non-homogeneous boundary condition occurs on the circle in the middle.
    \item \texttt{L-shape}: The computational domain is the L-shaped domain, and it is discretized with quasi-uniform triangles.
    The velocity field and pressure are chosen as
    $\bu(x,y)=(\sin(\pi x)\sin(\pi y),\cos(\pi x)\cos(\pi y))$ and $p(x,y) = |y|$.
\end{itemize}
Figure \ref{figure: para-meshtype} shows the above meshes, the corresponding mesh qualities, and the velocity errors with different penalty parameters.
The mesh quality of a triangle \cite{chen2004mesh} is defined as the ratio of its area to the sum of the squares of its sides, which implies that the equilateral triangle has the best mesh quality 1 and sharper triangles are closer to 0.
For the \texttt{EG} method, sufficiently large penalty parameters are required to achieve a desired accuracy.
Moreover, such large parameters are depending on the meshes.
To be specific, in \texttt{Perturbed mesh}, some bad quality triangles cause an unexpected spike around $\rho=3$ in the velocity errors, which makes it more difficult to choose a proper penalty parameter.
However, on all the given meshes, the \texttt{mEG} method seems uniformly stable with any positive penalty parameter.
The \texttt{mEG} method, even for the mesh with bad quality triangles, has good performance.
These numerical results confirm that the \texttt{mEG} method is a parameter-free scheme.

% 6.1.2

\subsubsection{Pressure-robustness test}

In this test, we verify the pressure-robustness of the \texttt{PR-mEG} method.
We solve the example problem \eqref{eqn: example1} with varying $\nu$, from $10^{-2}$ to $10^{-6}$, to confirm the error behaviors expected in \eqref{sys: errboundst} and \eqref{sys: errboundpr}.
The mesh size is fixed as $h=1/32$.
\begin{figure}[h!]
    \centering
    \includegraphics[width=.35\textwidth]{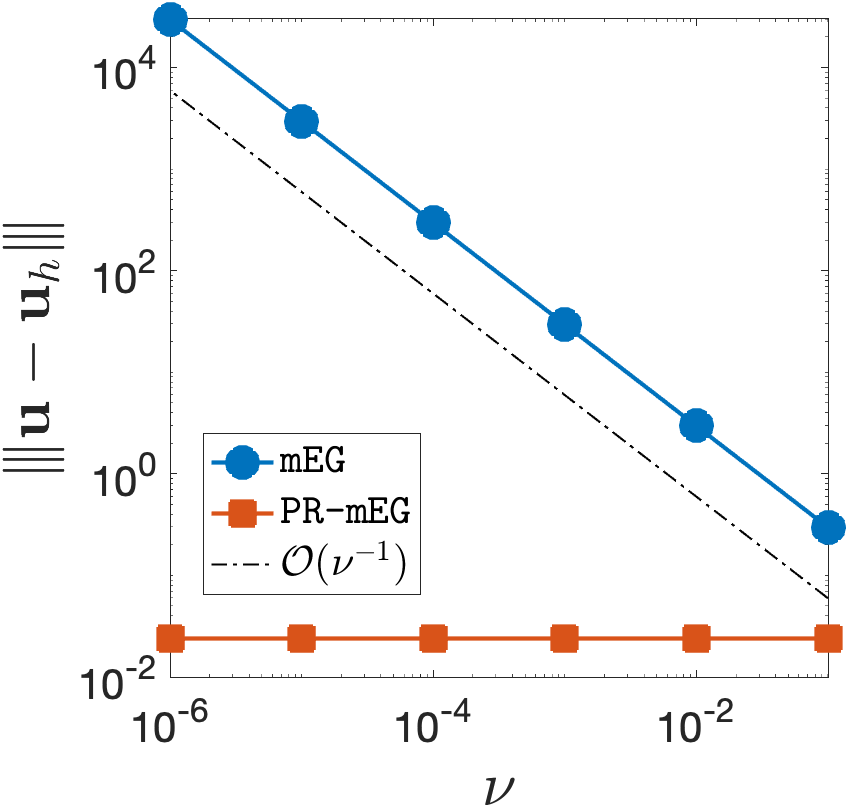}
    \hskip 20pt
    \includegraphics[width=.35\textwidth]{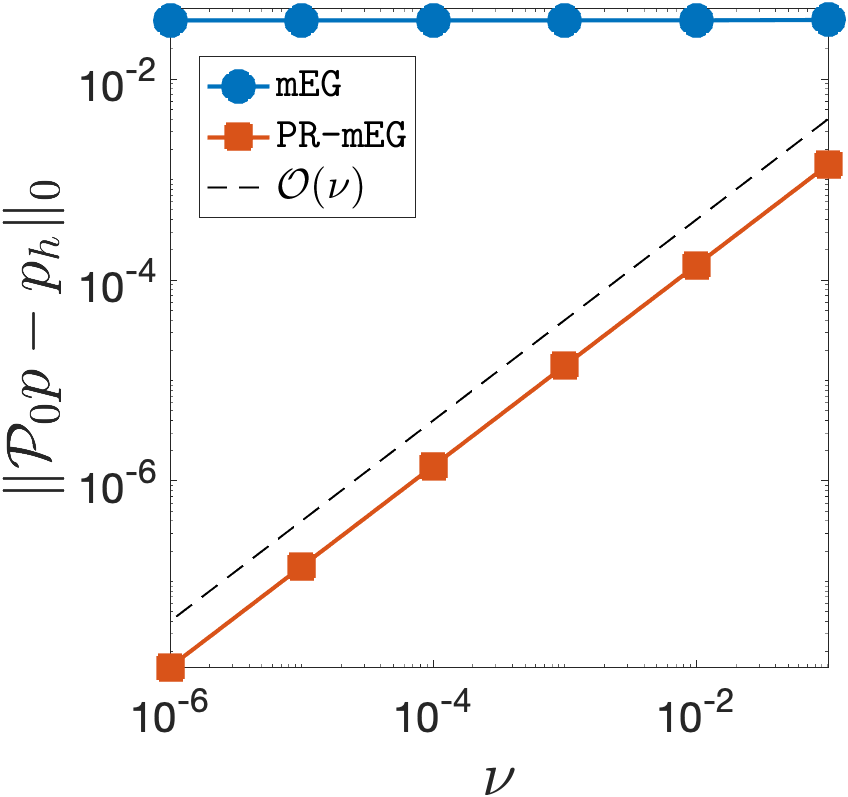}
    \caption{Error profiles of the \texttt{mEG} and \texttt{PR-mEG} methods with varying $\nu$ values and a fixed mesh size $h=1/32$.}
    \label{figure: mEGvsPRmEG}
\end{figure}
Figure~\ref{figure: mEGvsPRmEG} shows the velocity errors $\trinorm{\bu-\bu_h}$ and  pressure errors $\norm{\mathcal{P}_0p-p_h}_0$ of the \texttt{mEG} and \texttt{PR-mEG} methods.
In Figure~\ref{figure: mEGvsPRmEG}, the \texttt{mEG} method produces the velocity errors proportional to $\nu^{-1}$ because the second term $h\nu^{-1}\norm{p}_1$ of the error bound \eqref{eqn: errboundstp} becomes dominant as $\nu$ gets smaller.
Also, since the pressure error $\norm{\mathcal{P}_0p-p_h^{\texttt{mEG}}}_0$ is bounded by a dominant term $h\norm{p}_1$, the error remains the same.
On the other hand, the \texttt{PR-mEG} method produces the same velocity errors regardless of $\nu$, and its pressure errors decrease in proportion to $\nu$.
These numerical results support our theoretical error estimates related to the pressure-robustness in \eqref{sys: errboundst} and \eqref{sys: errboundpr}.

\begin{table}[h!]
    \centering
    \begin{tabular}{|c||c|c||c|c|}
    \hline
        &  \multicolumn{2}{c||}{\texttt{mEG}} & 
         \multicolumn{2}{c|}{\texttt{PR-mEG}} \\
    \cline{2-5}   
       $h$  & { $\trinorm{\mathbf{u}-\mathbf{u}_h^{\texttt{mEG}}}$} & {Rate} &
       { $\trinorm{\mathbf{u}-\mathbf{u}_h^{\texttt{PR}}}$} & {Rate} \\
       \hline
       $1/8$ & 2.577e+5 & - &  9.727e-2 & - \\
       \hline
       $1/16$  & 9.097e+4 & 1.50 & 4.749e-2 & 1.03  \\
       \hline
       $1/32$   & 3.183e+4 & 1.52 & 2.339e-2 & 1.02  \\
       \hline
       $1/64$  & 1.116e+4 & 1.51 & 1.159e-2 & 1.01  \\
       \hline
       \hline
        $h$ & {$\|p-p_h^{\texttt{mEG}}\|_0$} & {Rate} &
       {$\|p-p_h^{\texttt{PR}}\|_0$} & {Rate}  \\ 
       \hline
       $1/8$ & 5.736e-1 & - &  4.802e-1 & -  \\
       \hline
       $1/16$  & 2.694e-1 & 1.09 & 2.404e-1 & 1.00 \\
       \hline
       $1/32$   & 1.310e-1 & 1.04 & 1.203e-1 & 1.00  \\
       \hline
       $1/64$  & 6.464e-2 & 1.02 & 6.014e-2 & 1.00 \\
       \hline
    \end{tabular}
    \caption{A mesh refinement study for \texttt{mEG} and \texttt{PR-mEG} with varying mesh size $h$ and $\nu=10^{-6}$.}
    \label{table: mEGvsPRmEG}
\end{table}
Furthermore, we perform a mesh refinement study for the \texttt{mEG} and \texttt{PR-mEG} methods with decreasing mesh size $h$ and fixed $\nu = 10^{-6}$.
As shown in Table~\ref{table: mEGvsPRmEG}, the velocity and pressure errors for both methods decrease in at least the first order of convergence, and the pressure errors look very similar in magnitude.
However, even though the velocity errors for the \texttt{mEG} method decrease at a faster rate, the magnitude of the errors seems huge. 
Thus, it may not be possible to obtain accurate numerical velocity from the \texttt{mEG} method unless $h$ is small enough.
On the other hand, the \texttt{PR-mEG} method yields about a million times smaller velocity errors than the \texttt{mEG} method.
This means that the \texttt{PR-mEG} method provides a significantly improved numerical velocity for the Stokes equations with small viscosity, which is an important feature of pressure-robust numerical schemes.

%%%%%%%%%% 6.2 %%%%%%%%%%%

\subsection{Three dimensional examples}
\label{3dexample}

We consider a 3D flow in a unit cube $\Omega=(0,1)^3$. The velocity field and pressure are chosen as
\begin{equation}
    \bu 
    = \left(\begin{array}{c}
    \sin(\pi x)\cos(\pi y) - \sin(\pi x)\cos(\pi z) \\
    \sin(\pi y)\cos(\pi z) - \sin(\pi y)\cos(\pi x) \\
    \sin(\pi z)\cos(\pi x) - \sin(\pi z)\cos(\pi y)
    \end{array}\right),\quad
    p = \sin(\pi x)\sin(\pi y)\sin(\pi z).
    \label{eqn: example2}
\end{equation}

% 6.2.1

\subsubsection{Parameter-free test}
In this example, with different penalty parameters, we compute the velocity and pressure errors and condition numbers in the \texttt{EG} and \texttt{mEG} methods.
\begin{figure}[h!]
    \centering
    \includegraphics[width=.3\textwidth]{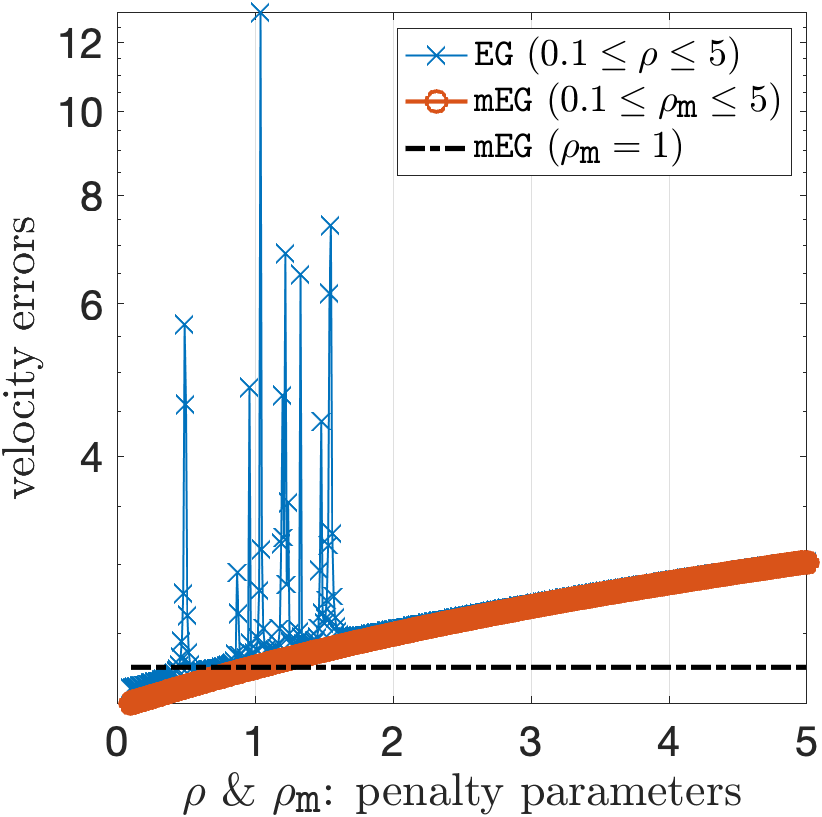}
    \hskip 5pt
    \includegraphics[width=.3\textwidth]{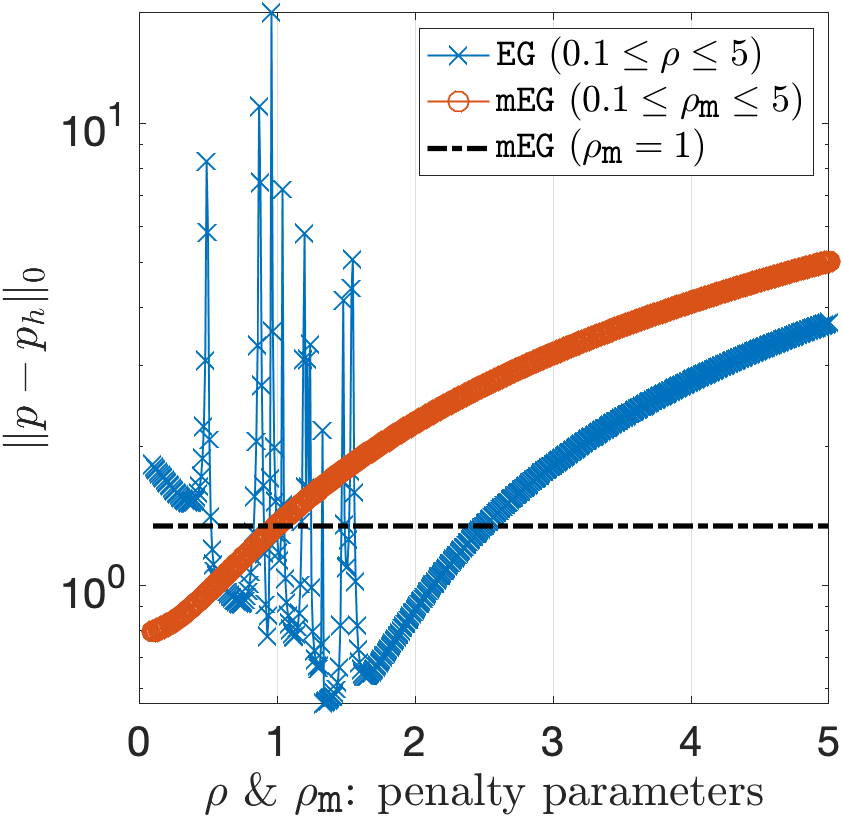}
    \hskip 5pt
    \includegraphics[width=.3\textwidth]{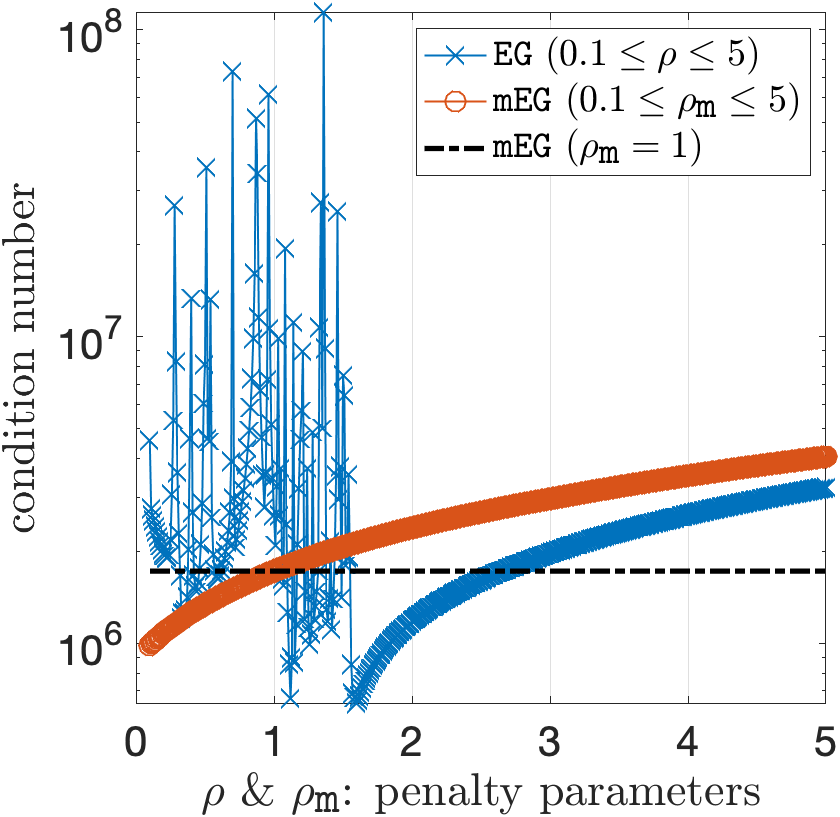}
    \caption{Errors and condition numbers of \texttt{EG} and \texttt{mEG} for $0.1\leq\rho\leq 5$ ($\nu=1$, $h=1/4$) in the 3D case.}
    \label{figure: 3Dpara-errors}
\end{figure}
Figure~\ref{figure: 3Dpara-errors} clearly shows the need of a sufficiently large penalty parameter for the \texttt{EG} method and the stability of the \texttt{mEG} method with any positive parameter.
\begin{figure}[h!]
\centering
\begin{tabular}{cc}
    \includegraphics[width=.35\textwidth]{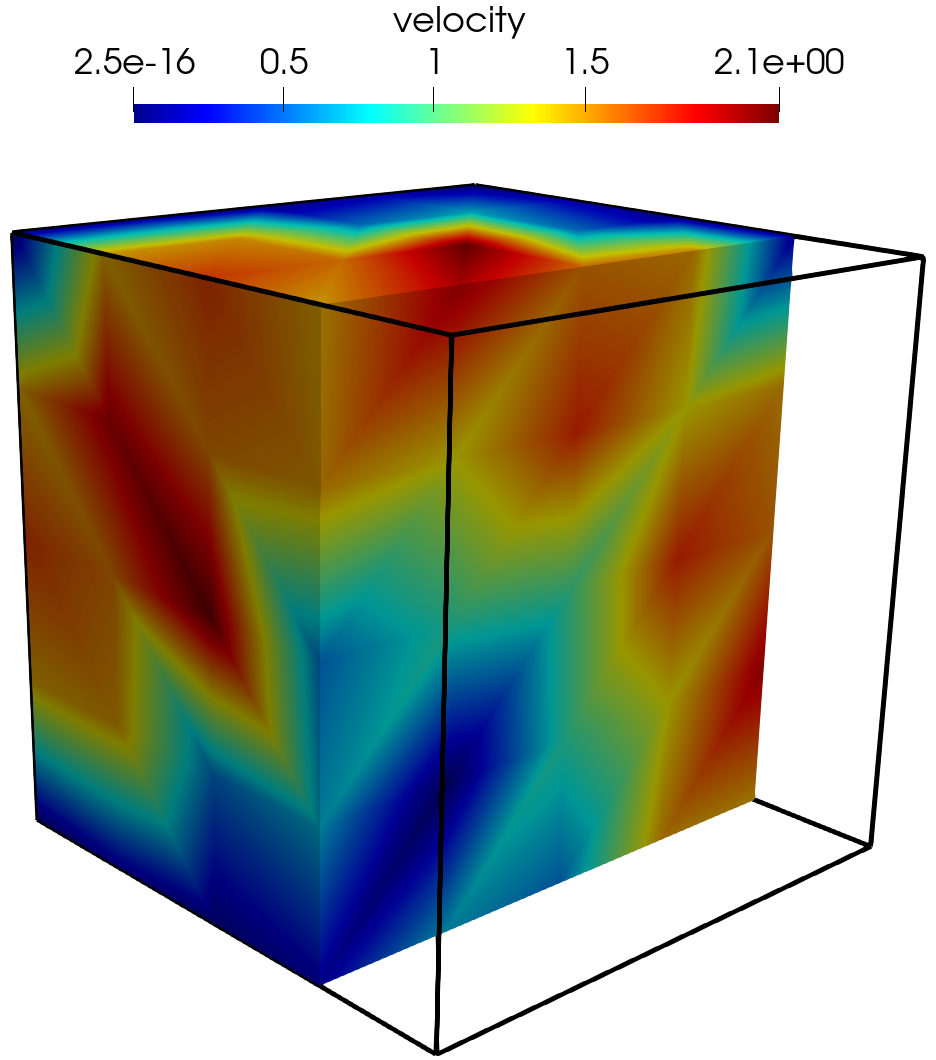}
   &\includegraphics[width=.35\textwidth]{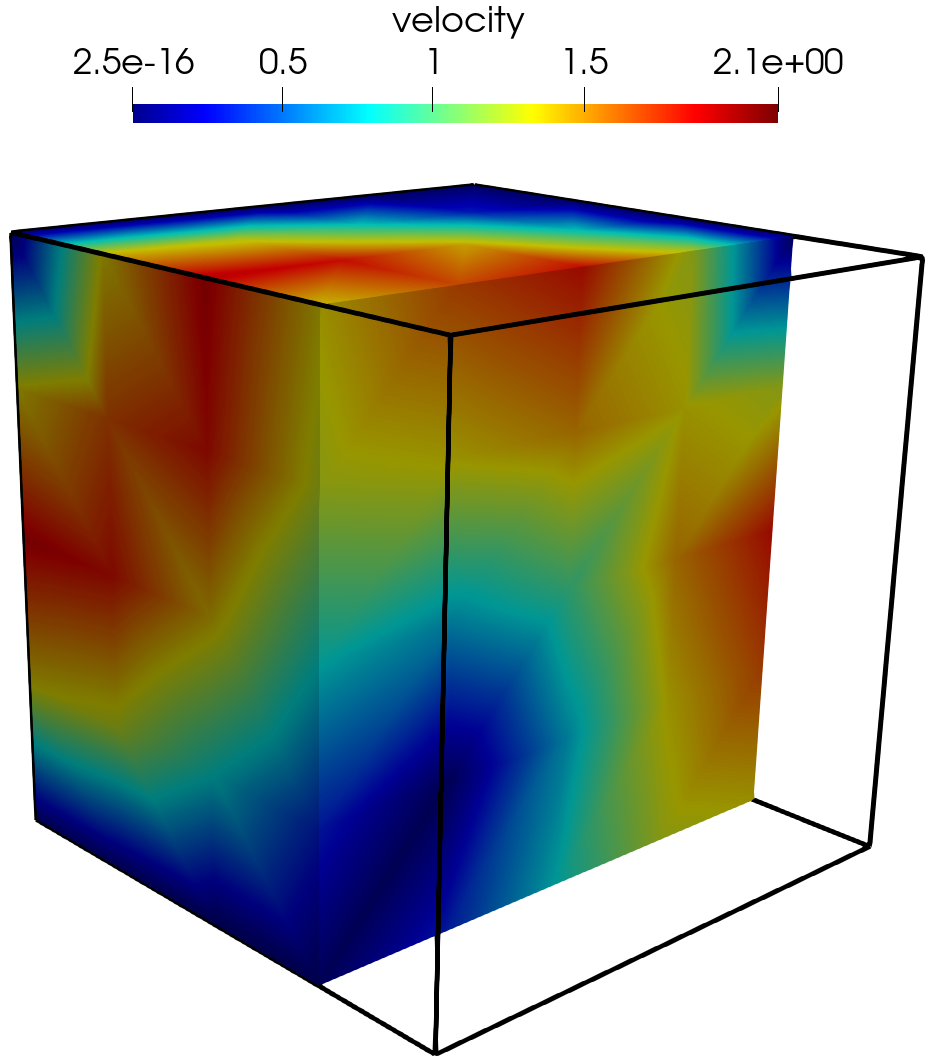}\\
   \texttt{EG} ($\rho=1$)  &\texttt{mEG} ($\rho_\texttt{m}=1$)
\end{tabular}
    \caption{Comparison of 3D numerical velocity solutions when $h=1/4$ and $\nu=1$.}
    \label{figure: 3d4_EGvsmEG_u}
\end{figure}
Figure~\ref{figure: 3d4_EGvsmEG_u} also displays the magnitudes of the numerical velocity of the \texttt{EG} and \texttt{mEG} methods with $\rho=\rho_\texttt{m}=1$.
Although both methods give similar color patterns, some sharp changes in color occur in the velocity of the \texttt{EG} method.
The sharp changes are caused by the insufficiently large parameter, and they make the numerical velocity inaccurate.

In addition, we focus on the effect of large penalty parameters on the errors and condition numbers.
In Figure~\ref{figure: 3Dpara-errors}, the condition numbers of both methods tend to increase with the parameters $\rho$ and $\rho_\texttt{m}$, which causes the increased velocity and pressure errors.
\begin{table}[ht]
    \centering
    \begin{tabular}{|c||c|c||c|c||c|c|}
    \hline
        &  \multicolumn{2}{c||}{\texttt{EG} ($\rho=10$)
        } & 
         \multicolumn{2}{c||}{\texttt{EG} ($\rho=2$)
        } &
        \multicolumn{2}{c|}{\texttt{mEG} ($\rho_{\texttt{m}}=1$)
        }\\
    \cline{2-7}   
       $h$  & {$\|\mathbf{u}-\mathbf{u}_h^{\texttt{EG}}\|_\mathcal{E}$} & {Rate} &
       {$\|\mathbf{u}-\mathbf{u}_h^{\texttt{EG}}\|_\mathcal{E}$} & Rate &
       {$\trinorm{\mathbf{u}-\mathbf{u}_h^{\texttt{mEG}}}$} & { Rate} \\ 
       \hline
       $1/4$ & 3.719e+0 & - &  2.518e+0 & - & 2.284e+0 & - \\
       \hline
       $1/8$  & 1.827e+0 & 1.03 & 1.228e+0 & 1.04 & 1.121e+0 & 1.03  \\
       \hline
       $1/16$   & 9.048e-1 & 1.01 & 6.052e-1 & 1.02 & 5.552e-1 & 1.01  \\
       \hline
       $1/32$  & 4.501e-1 & 1.01 & 3.007e-1 & 1.01 & 2.764e-1 & 1.01 \\
       \hline
       \hline
        $h$ & {$\|p-p_h^{\texttt{EG}}\|_0$} & {Rate} &
       {$\|p-p_h^{\texttt{EG}}\|_0$} & {Rate} &
       {$\|p-p_h^{\texttt{mEG}}\|_0$} & {Rate} \\ 
       \hline
       $1/4$ & 8.377e+0 & - &  8.819e-1 & - & 1.349e+0 & - \\
       \hline
       $1/8$  & 3.600e+0 & 1.22 & 3.611e-1 & 1.29 & 6.098e-1 & 1.15 \\
       \hline
       $1/16$   & 1.670e+0 & 1.11 & 1.688e-1 & 1.10 & 3.011e-1 & 1.02 \\
       \hline
       $1/32$  & 8.312e-1 & 1.01 & 8.411e-2 & 1.00 & 1.504e-1 & 1.00 \\
       \hline
    \end{tabular}
    \caption{A mesh refinement study for \texttt{EG} and \texttt{mEG} with varying mesh size $h$ and $\nu=1$ in the 3D case.}
    \label{table: 3dEGvsmEG}
\end{table}
In order to perform quantitative comparison, we choose $\rho=10$ and $\rho=2$ based on the results in Figure~\ref{figure: 3Dpara-errors} and report the pressure and velocity errors of the two cases in Table~\ref{table: 3dEGvsmEG}.
The pressure errors of the \texttt{EG} method with $\rho=10$ are 10 times bigger than those with $\rho=2$, even though their pressure errors decrease at the same rate.
Thus, in practice, a penalty parameter $\rho$ cannot be chosen too large due to this accuracy issue.
It may also be challenging to choose a proper $\rho$ because it varied with meshes.
On the other hand, 
for the \texttt{mEG} method with $\rho_\texttt{m}=1$, the convergence rates of the velocity and pressure errors are of at least first-order, and the \texttt{mEG} method yields smaller velocity errors than the \texttt{EG} method with $\rho=2$.
Therefore, with the \texttt{mEG} method, we can always safely choose $\rho_\texttt{m}=1$ to achieve reliable performance, so we can make the simulation lower-cost.

% 6.2.2

\subsubsection{Pressure-robust test}

To verify the pressure-robustness in the three-dimensional example \eqref{eqn: example2}, we consider the pattern of the error behaviors obtained from the \texttt{mEG} and \texttt{PR-mEG} methods when $\nu$ varies and the mesh size is fixed to $h = 1/16$.
\begin{figure}[h!]
    \centering
    \includegraphics[width=.35\textwidth]{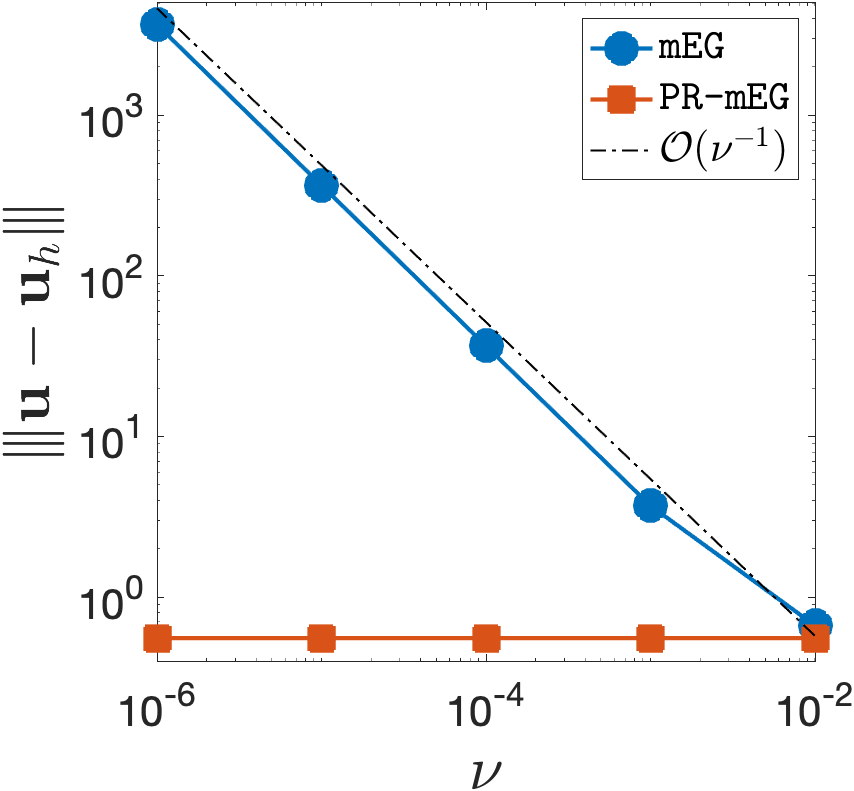}
    \hskip 20pt
    \includegraphics[width=.35\textwidth]{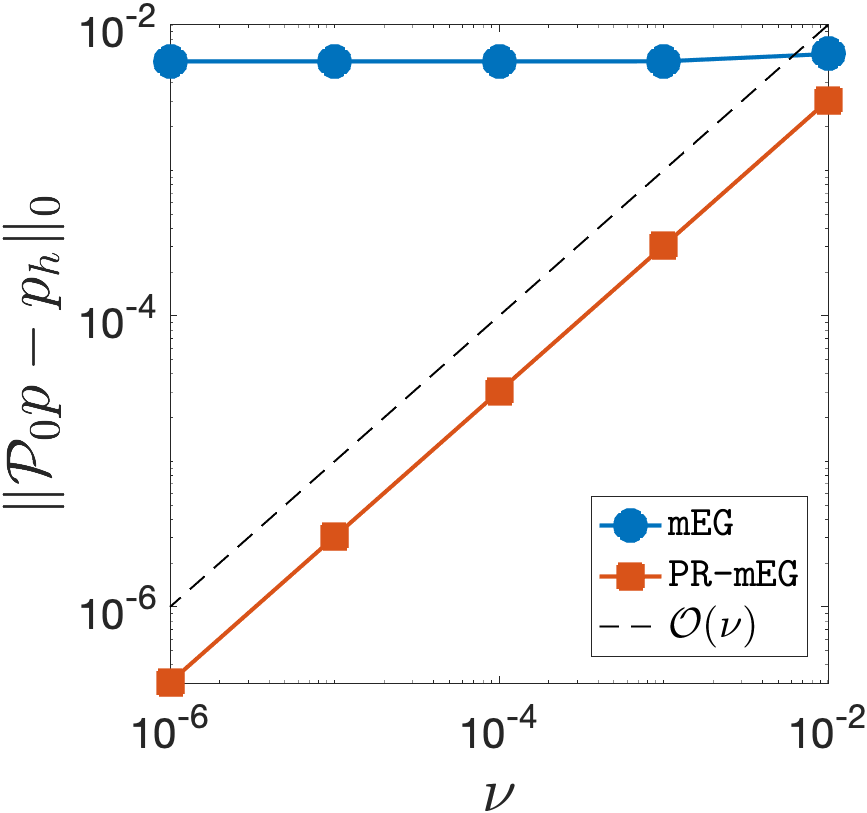}
    \caption{Error profiles of the \texttt{mEG} and \texttt{PR-mEG} methods with varying $\nu$ values and a fixed mesh size $h=1/16$ in the 3D case.}
    \label{figure: 3dmEGvsPRmEG}
\end{figure}
In Figure~\ref{figure: 3dmEGvsPRmEG}, we observe the same error behaviors as those in the two dimensional pressure-robust test.
That is, for the \texttt{mEG} method, the velocity errors are inversely proportional to $\nu$ while the pressure errors tend to stay constant.
On the other hand, for the \texttt{PR-mEG} method, the velocity errors seem independent of $\nu$, and the pressure errors decrease in proportion to $\nu$.

\begin{figure}[h!]
\centering
\begin{tabular}{cc}
\includegraphics[width=.35\textwidth]{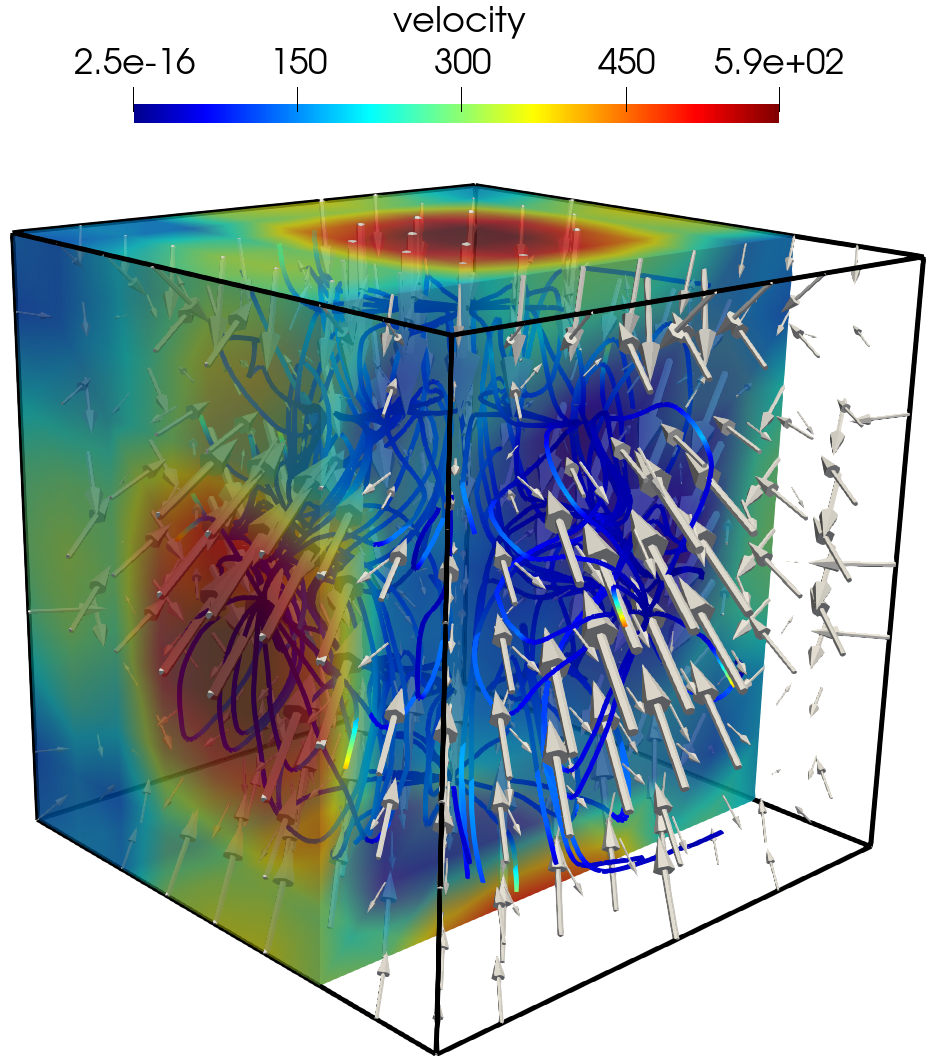}
&\includegraphics[width=.35\textwidth]{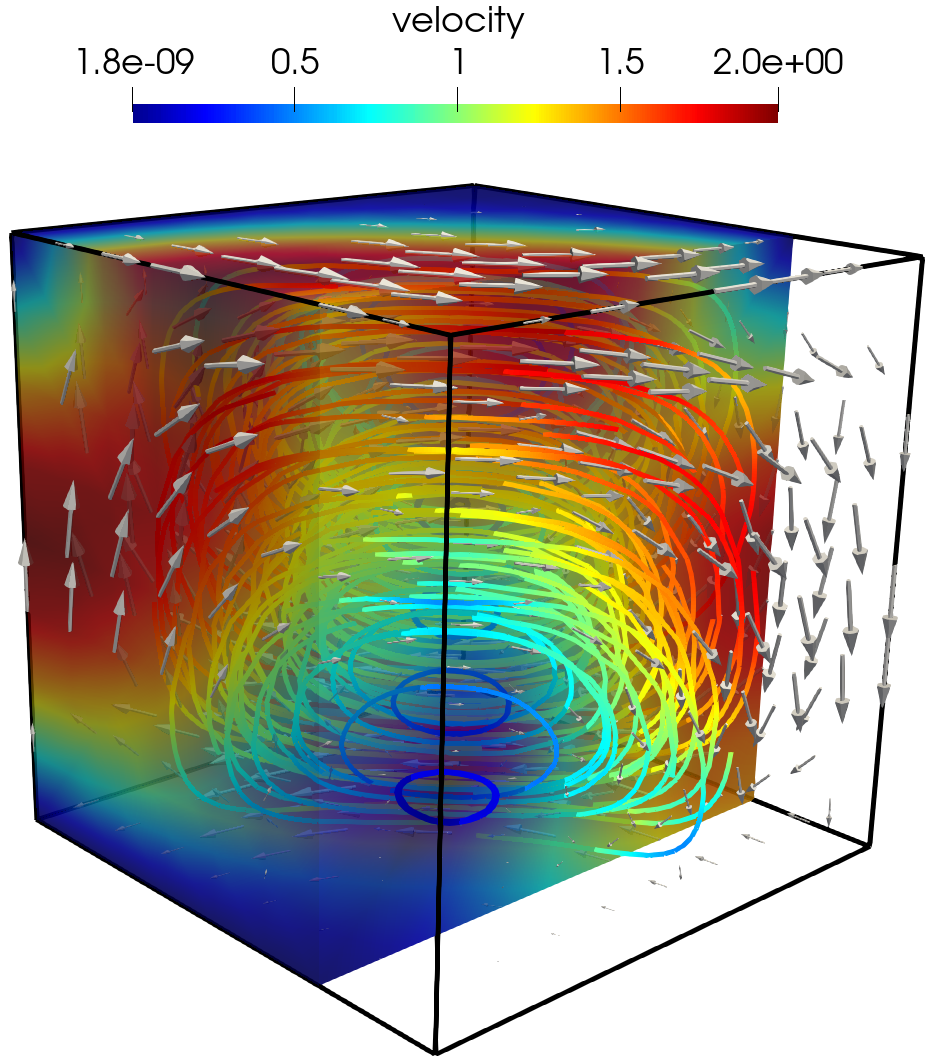}\\
\texttt{mEG}&\texttt{PR-mEG}
\end{tabular}
    \caption{Streamlines of the 3D numerical velocity when $h=1/8$ and $\nu=10^{-6}$.}
    \label{figure: 3dnumvelocity}
\end{figure}
Furthermore, Figure~\ref{figure: 3dnumvelocity} shows the streamlines of the velocity solutions of the \texttt{mEG} and \texttt{PR-mEG} methods when $h=1/8$ and $\nu=10^{-6}$.
In this case, the velocity error of the \texttt{mEG} method is 1.014e+4 while that of the \texttt{PR-mEG} method is 1.122e$+$0.
As shown in Figure~\ref{figure: 3dnumvelocity}, the numerical velocity of the \texttt{PR-mEG} method well captures the 3D vortex flow, while that of the \texttt{mEG} method is unable to do so.

%%%%%%%%%%%%%%%%%%%%%%%%%%%%%%%%%%
% 7. Conclusion
%%%%%%%%%%%%%%%%%%%%%%%%%%%%%%%%%%

\section{Conclusion}
\label{sec:conclusion}
In this paper, we proposed a low-cost, parameter-free, and pressure-robust Stokes solver based on the EG method operating with minimal degrees of freedom.
The weak derivatives, simply computed by the geometric data of elements, allowed the EG method to be free of penalty parameters and some IPDG trace terms.
With reduced computational complexity, the modified EG method preserved the minimal degrees of freedom and the optimal rates in convergence of the EG method.
Furthermore, we achieved the pressure-robustness for the modified EG method by the simple modification on the right-hand side.
We also confirmed the improved theoretical results through the several numerical tests with two- and three-dimensional examples.
The idea of using weak derivatives can be applied to improve other numerical schemes employing the symmetric IPDG formulation.
The extension of the idea to numerical schemes for the biharmonic equation will be one of our future research directions.
We expect that the weak derivatives corresponding to the biharmonic equation provide significant computational advantages in solving application problems involving the biharmonic equation numerically.

%	\newpage
	\bibliography{Stokes}

\begin{thebibliography}{35}
\providecommand{\natexlab}[1]{#1}
\providecommand{\url}[1]{\texttt{#1}}
\expandafter\ifx\csname urlstyle\endcsname\relax
  \providecommand{\doi}[1]{doi: #1}\else
  \providecommand{\doi}{doi: \begingroup \urlstyle{rm}\Url}\fi

\bibitem[Ainsworth(2007)]{ainsworth2007posteriori}
Mark Ainsworth.
\newblock A posteriori error estimation for discontinuous {G}alerkin finite
  element approximation.
\newblock \emph{SIAM Journal on Numerical Analysis}, 45\penalty0 (4):\penalty0
  1777--1798, 2007.

\bibitem[Ainsworth and Rankin(2010)]{ainsworth2010fully}
Mark Ainsworth and Richard Rankin.
\newblock Fully computable error bounds for discontinuous {G}alerkin finite
  element approximations on meshes with an arbitrary number of levels of
  hanging nodes.
\newblock \emph{SIAM Journal on Numerical Analysis}, 47\penalty0 (6):\penalty0
  4112--4141, 2010.

\bibitem[Babu{\v{s}}ka(1973)]{Babuska73}
Ivo Babu{\v{s}}ka.
\newblock The finite element method with {L}agrangian multipliers.
\newblock \emph{Numerische Mathematik}, 20\penalty0 (3):\penalty0 179--192,
  1973.

\bibitem[Bernardi and Raugel(1985)]{BernardiRaugel85}
Christine Bernardi and Genevieve Raugel.
\newblock Analysis of some finite elements for the {S}tokes problem.
\newblock \emph{Mathematics of Computation}, 44\penalty0 (169):\penalty0
  71--79, 1985.

\bibitem[Brenner and Sung(2005)]{brenner2005c}
Susanne~C Brenner and Li-Yeng Sung.
\newblock ${C}^0$ interior penalty methods for fourth order elliptic boundary
  value problems on polygonal domains.
\newblock \emph{Journal of Scientific Computing}, 22\penalty0 (1):\penalty0
  83--118, 2005.

\bibitem[Brezzi(1974)]{Brezzi74}
Franco Brezzi.
\newblock On the existence, uniqueness and approximation of saddle-point
  problems arising from {L}agrangian multipliers.
\newblock \emph{Publications math{\'e}matiques et informatique de Rennes},
  \penalty0 (S4):\penalty0 1--26, 1974.

\bibitem[Chaabane et~al.(2018)Chaabane, Girault, Rivi{\`e}re, and
  Thompson]{ChaabaneEtAl18}
Nabil Chaabane, Vivette Girault, B{\'e}atrice Rivi{\`e}re, and Travis Thompson.
\newblock A stable enriched {G}alerkin element for the {S}tokes problem.
\newblock \emph{Applied Numerical Mathematics}, 132:\penalty0 1--21, 2018.

\bibitem[Chen(2004)]{chen2004mesh}
Long Chen.
\newblock Mesh smoothing schemes based on optimal {D}elaunay triangulations.
\newblock In \emph{IMR}, pages 109--120, 2004.

\bibitem[Chen(2009)]{CHE09}
Long Chen.
\newblock \emph{iFEM: An Integrated Finite Element Methods Package in MATLAB}.
\newblock Technical Report, University of California at Irvine, 2009.

\bibitem[Cockburn et~al.(2009)Cockburn, Gopalakrishnan, and
  Lazarov]{cockburn2009unified}
Bernardo Cockburn, Jayadeep Gopalakrishnan, and Raytcho Lazarov.
\newblock Unified hybridization of discontinuous {G}alerkin, mixed, and
  continuous {G}alerkin methods for second order elliptic problems.
\newblock \emph{SIAM Journal on Numerical Analysis}, 47\penalty0 (2):\penalty0
  1319--1365, 2009.

\bibitem[Crouzeix and Raviart(1973)]{CrouzeixRaviart73}
Michel Crouzeix and P-A Raviart.
\newblock Conforming and nonconforming finite element methods for solving the
  stationary {S}tokes equations {I}.
\newblock \emph{Revue fran{\c{c}}aise d'automatique informatique recherche
  op{\'e}rationnelle. Math{\'e}matique}, 7\penalty0 (R3):\penalty0 33--75,
  1973.

\bibitem[Di~Pietro and Ern(2015)]{di2015hybrid}
Daniele~A Di~Pietro and Alexandre Ern.
\newblock Hybrid high-order methods for variable-diffusion problems on general
  meshes.
\newblock \emph{Comptes Rendus Math{\'e}matique}, 353\penalty0 (1):\penalty0
  31--34, 2015.

\bibitem[Epshteyn and Rivi{\`e}re(2007)]{epshteyn2007estimation}
Yekaterina Epshteyn and B{\'e}atrice Rivi{\`e}re.
\newblock Estimation of penalty parameters for symmetric interior penalty
  {G}alerkin methods.
\newblock \emph{Journal of Computational and Applied Mathematics}, 206\penalty0
  (2):\penalty0 843--872, 2007.

\bibitem[Gauger et~al.(2019)Gauger, Linke, and Schroeder]{GaugerLS19}
Nicolas~R Gauger, Alexander Linke, and Philipp~W Schroeder.
\newblock On high-order pressure-robust space discretisations, their advantages
  for incompressible high {R}eynolds number generalised {B}eltrami flows and
  beyond.
\newblock \emph{The SMAI Journal of Computational Mathematics}, 5:\penalty0
  89--129, 2019.

\bibitem[Girault et~al.(2005)Girault, Rivi{\`e}re, and Wheeler]{GiraultEtAl05}
Vivette Girault, B{\'e}atrice Rivi{\`e}re, and Mary~F. Wheeler.
\newblock A discontinuous {G}alerkin method with nonoverlapping domain
  decomposition for the {S}tokes and {N}avier-{S}tokes problems.
\newblock \emph{Mathematics of Computation}, 74\penalty0 (249):\penalty0
  53--84, 2005.

\bibitem[Hansbo and Larson(2008)]{HansboLarson08}
Peter Hansbo and Mats~G. Larson.
\newblock Piecewise divergence-free discontinuous {G}alerkin methods for
  {S}tokes flow.
\newblock \emph{Communications in Numerical Methods in Engineering},
  24\penalty0 (5):\penalty0 355--366, 2008.

\bibitem[Hu et~al.(2022)Hu, Lee, Mu, and Yi]{HuLeeMuYi}
Xiaozhe Hu, Seulip Lee, Lin Mu, and Son-Young Yi.
\newblock Pressure-robust enriched {G}alerkin methods for the {S}tokes
  equations.
\newblock \emph{arXiv:2208.13076}, 2022.

\bibitem[Ladyzhenskaya(1969)]{Ladyzhenskaya69}
Olga~Aleksandrovna Ladyzhenskaya.
\newblock \emph{The Mathematical Theory of Viscous Incompressible Flow},
  volume~2.
\newblock Gordon and Breach New York, 1969.

\bibitem[Lee et~al.(2016)Lee, Lee, and Wheeler]{lee2016locally}
Sanghyun Lee, Young-Ju Lee, and Mary~F Wheeler.
\newblock A locally conservative enriched {G}alerkin approximation and
  efficient solver for elliptic and parabolic problems.
\newblock \emph{SIAM Journal on Scientific Computing}, 38\penalty0
  (3):\penalty0 A1404--A1429, 2016.

\bibitem[Li and Zikatanov(2022)]{LiZikatanov22}
Yuwen Li and Ludmil~T Zikatanov.
\newblock New stabilized ${P}_1\times {P}_0$ finite element methods for nearly
  inviscid and incompressible flows.
\newblock \emph{Computer Methods in Applied Mechanics and Engineering},
  393:\penalty0 114815, apr 2022.

\bibitem[Linke(2012)]{Linke12}
Alexander Linke.
\newblock A divergence-free velocity reconstruction for incompressible flows.
\newblock \emph{Comptes Rendus Mathematique}, 350\penalty0 (17-18):\penalty0
  837--840, 2012.

\bibitem[Linke and Merdon(2016)]{LinkeMerdon16}
Alexander Linke and Christian Merdon.
\newblock Pressure-robustness and discrete {H}elmholtz projectors in mixed
  finite element methods for the incompressible {N}avier--{S}tokes equations.
\newblock \emph{Computer Methods in Applied Mechanics and Engineering},
  311:\penalty0 304--326, 2016.

\bibitem[Mu(2020)]{Mu20}
Lin Mu.
\newblock Pressure robust weak {G}alerkin finite element methods for {S}tokes
  problems.
\newblock \emph{SIAM Journal on Scientific Computing}, 42\penalty0
  (3):\penalty0 B608--B629, 2020.

\bibitem[Mu et~al.(2015)Mu, Wang, and Ye]{mu2015modified}
Lin Mu, Xiaoshen Wang, and Xiu Ye.
\newblock A modified weak {G}alerkin finite element method for the {S}tokes
  equations.
\newblock \emph{Journal of Computational and Applied Mathematics},
  275:\penalty0 79--90, 2015.

\bibitem[Mu et~al.(2018)Mu, Wang, Ye, and Zhang]{MuWYZ18}
Lin Mu, Junping Wang, Xiu Ye, and Shangyou Zhang.
\newblock A discrete divergence free weak {G}alerkin finite element method for
  the {S}tokes equations.
\newblock \emph{Applied Numerical Mathematics}, 125:\penalty0 172--182, 2018.

\bibitem[Mu et~al.(2021)Mu, Ye, and Zhang]{MuYZ21(2)}
Lin Mu, Xiu Ye, and Shangyou Zhang.
\newblock Development of pressure-robust discontinuous {G}alerkin finite
  element methods for the {S}tokes problem.
\newblock \emph{Journal of Scientific Computing}, 89\penalty0 (1):\penalty0
  1--25, 2021.

\bibitem[Sun and Liu(2009)]{sun2009locally}
Shuyu Sun and Jiangguo Liu.
\newblock A locally conservative finite element method based on piecewise
  constant enrichment of the continuous {G}alerkin method.
\newblock \emph{SIAM Journal on Scientific Computing}, 31\penalty0
  (4):\penalty0 2528--2548, 2009.

\bibitem[Taylor and Hood(1973)]{TaylorHood73}
C.~Taylor and P.~Hood.
\newblock A numerical solution of the {N}avier-{S}tokes equations using the
  finite element technique.
\newblock \emph{Computers \& Fluids}, 1\penalty0 (1):\penalty0 73--100, 1973.

\bibitem[Wang and Ye(2013)]{wang2013weak}
Junping Wang and Xiu Ye.
\newblock A weak {G}alerkin finite element method for second-order elliptic
  problems.
\newblock \emph{Journal of Computational and Applied Mathematics},
  241:\penalty0 103--115, 2013.

\bibitem[Wang and Ye(2016)]{wang2016weak}
Junping Wang and Xiu Ye.
\newblock A weak {G}alerkin finite element method for the {S}tokes equations.
\newblock \emph{Advances in Computational Mathematics}, 42\penalty0
  (1):\penalty0 155--174, 2016.

\bibitem[Wang et~al.(2014)Wang, Malluwawadu, Gao, and
  McMillan]{wang2014modified}
Xiaoshen Wang, Nolisa~S Malluwawadu, Fuzheng Gao, and TC~McMillan.
\newblock A modified weak {G}alerkin finite element method.
\newblock \emph{Journal of Computational and Applied Mathematics},
  271:\penalty0 319--327, 2014.

\bibitem[Xie et~al.(2020)Xie, Cao, Chen, and Zhong]{xie2020convergence}
Yingying Xie, Shuhao Cao, Long Chen, and Liuqiang Zhong.
\newblock Convergence and optimality of an adaptive modified weak {G}alerkin
  finite element method.
\newblock \emph{arXiv:2007.12853}, 2020.

\bibitem[Yi et~al.(2022{\natexlab{a}})Yi, Hu, Lee, and Adler]{YiEtAl22-Stokes}
Son-Young Yi, Xiaozhe Hu, Sanghyun Lee, and James~H. Adler.
\newblock An enriched {G}alerkin method for the {S}tokes equations.
\newblock \emph{Computers and Mathematics with Applications}, 120:\penalty0
  115--131, 2022{\natexlab{a}}.

\bibitem[Yi et~al.(2022{\natexlab{b}})Yi, Lee, and Zikatanov]{yi2022locking}
Son-Young Yi, Sanghyun Lee, and Ludmil~T Zikatanov.
\newblock Locking-free enriched {G}alerkin method for linear elasticity.
\newblock \emph{SIAM Journal on Numerical Analysis}, 60\penalty0 (1):\penalty0
  52--75, 2022{\natexlab{b}}.
\newblock ISSN 0036-1429.

\bibitem[Zhao et~al.(2022)Zhao, Park, and Chung]{zhao2020pressure}
Lina Zhao, Eun-Jae Park, and Eric Chung.
\newblock A pressure robust staggered discontinuous {G}alerkin method for the
  {S}tokes equations.
\newblock \emph{Computers \& Mathematics with Applications}, 128:\penalty0
  163--179, 2022.

\end{thebibliography}
	
\end{document}